\documentclass[sn-mathphys-num]{sn-jnl}

\usepackage{graphicx}
\usepackage{multirow}
\usepackage{amsmath,amssymb,amsfonts}
\usepackage{amsthm}
\usepackage{mathrsfs}
\usepackage[title]{appendix}
\usepackage{xcolor}
\usepackage{textcomp}
\usepackage{manyfoot}
\usepackage{booktabs}
\usepackage{algorithm}
\usepackage{algorithmic}
\usepackage{listings}

\usepackage{multirow,tabularx}
\usepackage{tikz}
\usepackage{cleveref}
\crefname{equation}{}{}
\crefname{figure}{figure}{figures}
\Crefname{figure}{Figure}{Figures}
\crefname{definition}{definition}{definitions}
\Crefname{definition}{Definition}{Definitions}
\crefname{convention}{concention}{conventions}
\Crefname{convention}{Concention}{Conventions}

\theoremstyle{thmstyleone}%
\newtheorem{theorem}{Theorem}
\newtheorem{lemma}{Lemma}
\newtheorem{convention}{Convention}
\newtheorem{assumption}{Assumption}

\newcommand{\TODO}[1][]{{\color{red}\textbf{TODO}\ifthenelse{\equal{#1}{}}{}{: #1}}}

\newcommand{\nelements}[1]{\left|#1\right|}

\newcommand{\closure}[1]{\overline{#1}}
\newcommand{\subentity}{E}
\newcommand{\refel}{R}
\newcommand{\cell}{C}
\newcommand{\polyspace}{\mathcal{V}}
\newcommand{\dualbasis}{\mathcal{L}}
\newcommand{\dual}[1]{#1^*}

\newcommand{\mat}[1]{\boldsymbol{\mathrm{#1}}}
\renewcommand{\vec}[1]{\boldsymbol{#1}}

\newcommand{\scurl}{\operatorname{curl}}
\newcommand{\sdiv}{\operatorname{div}}
\newcommand{\vcurl}{\scurl}

\newcommand{\mapping}{\mathcal{F}}
\newcommand{\geometrymap}{g}
\newcommand{\jacobian}{\mat{J}}

\newcommand{\rtriangle}[1][1.5mm]{\hspace{1pt}\begin{tikzpicture}[x=#1,y=#1,line cap=round,line join=round]
    \draw (0,0) -- (1,0) -- (0,1) -- cycle;
\end{tikzpicture}}
\newcommand{\interval}{\hspace{1pt}\raisebox{0.8mm}{\begin{tikzpicture}[x=1.5mm,y=1.5mm,line cap=round,line join=round]
    \draw (0,0) -- (1,0);
\end{tikzpicture}}}
\newcommand{\quadrilateral}{\hspace{1pt}\begin{tikzpicture}[x=1.5mm,y=1.5mm,line cap=round,line join=round]
    \draw (0,0) -- (1,0) -- (1, 1) -- (0,1) -- cycle;
\end{tikzpicture}}
\newcommand{\tetrahedron}{\hspace{1pt}\begin{tikzpicture}[x={(1.5mm,-0.2mm)},y={(1.3mm,0.5mm)},z={(0,1.5mm)},line cap=round,line join=round]
    \draw (0,0,0) -- (1,0,0) -- (0,1,0) -- cycle (0,0,1) -- (0,0,0) (0,0,1) -- (1,0,0) (0,0,1) -- (0,1,0);
\end{tikzpicture}}
\newcommand{\hexahedron}{\hspace{1pt}\begin{tikzpicture}[x={(1.5mm,-0.2mm)},y={(1.3mm,0.5mm)},z={(0,1.5mm)},line cap=round,line join=round]
    \draw (0,0,0) -- (1,0,0) -- (1,1,0) -- (0,1,0) -- cycle (0,0,1) -- (1,0,1) -- (1,1,1) -- (0,1,1) -- cycle (0,0,0) -- (0,0,1) (1,0,0) -- (1,0,1) (1,1,0) -- (1,1,1) (0,1,0) -- (0,1,1);
\end{tikzpicture}}
\newcommand{\prism}{\hspace{1pt}\begin{tikzpicture}[x={(1.5mm,-0.2mm)},y={(1.3mm,0.5mm)},z={(0,1.5mm)},line cap=round,line join=round]
    \draw (0,0,0) -- (1,0,0) -- (0,1,0) -- cycle (0,0,1) -- (1,0,1) -- (0,1,1) -- cycle (0,0,0) -- (0,0,1) (1,0,0) -- (1,0,1) (0,1,0) -- (0,1,1);
\end{tikzpicture}}
\newcommand{\pyramid}{\hspace{1pt}\begin{tikzpicture}[x={(1.5mm,-0.2mm)},y={(1.3mm,0.5mm)},z={(0,1.5mm)},line cap=round,line join=round]
    \draw (0,0,0) -- (1,0,0) -- (1,1,0) -- (0,1,0) -- cycle (0,0,0) -- (0,0,1) (1,0,0) -- (0,0,1) (1,1,0) -- (0,0,1) (0,1,0) -- (0,0,1);
\end{tikzpicture}}
\newcommand{\simplex}[1]{(\textup{simplex},#1)}
\newcommand{\hypercube}[1]{(\textup{cube},#1)}

\newcommand{\unctrace}[3]{\gamma^\perp(#1,#2,#3)}
\newcommand{\ctrace}[3]{\gamma^\parallel(#1,#2,#3)}

\newcommand{\transpose}{^T}
\newcommand{\invtranspose}{^{-T}}
\newcommand{\inv}{^{-1}}

\theoremstyle{thmstyletwo}%

\theoremstyle{thmstylethree}%
\newtheorem{definition}{Definition}%

\begin{document}

\title{DefElement: an encyclopedia of finite element definitions}

\author*[1]{\fnm{Matthew} W. \sur{Scroggs}\orcid{0000-0002-4658-2443}}\email{matthew.scroggs.14@ucl.ac.uk}
\author[2]{\fnm{Pablo} D. \sur{Brubeck}\orcid{0000-0002-3824-0080}}\email{brubeckmarti@maths.ox.ac.uk}
\author[3]{\fnm{Joseph} P. \sur{Dean}\orcid{0000-0001-7499-3373}}\email{jpd62@cam.ac.uk}
\author[4]{\fnm{J{\o}rgen} S. \sur{Dokken}\orcid{0000-0001-6489-8858}}\email{dokken@simula.no}
\author[2]{\fnm{India} \sur{Marsden}\orcid{0009-0006-0152-4780}}\email{marsden@maths.ox.ac.uk}

\affil*[1]{\orgdiv{Advanced Research Computing Centre}, \orgname{University College London}, \orgaddress{\street{Gower Street}, \city{London}, \postcode{WC1E 6BT}, \country{United Kingdom}}}

\affil[2]{\orgdiv{Department of Mathematics}, \orgname{Oxford University}, \orgaddress{\street{Woodstock Road}, \city{Oxford}, \postcode{OX2 6GG}, \country{United Kingdom}}}

\affil[3]{\orgdiv{Department of Engineering}, \orgname{University of Cambridge}, \orgaddress{\street{Trumpington Street}, \city{Cambridge}, \postcode{CB2 1PZ}, \country{United Kingdom}}}

\affil[4]{\orgname{Simula Research Laboratory}, \orgaddress{\street{Kristian Augusts gate 23}, \city{Oslo}, \postcode{0164}, \country{Norway}}}


\abstract{
DefElement is an online encyclopedia of finite element definitions that was
created and is maintained by the authors of this paper. DefElement aims to make
information about elements defined in the literature easily available in a
standard format. There are a number of open-source finite element libraries
available, and it can be difficult to check that an implementation of an
element in a library matches the element's definition in the literature or
implementation in another library, especially when many libraries include
variants of elements whose basis functions do not match exactly. In this paper,
we carefully derive conditions under which elements can be considered
equivalent and describe an algorithm that uses these conditions to verify that two
implementations of a finite element are indeed variants of the same element.
The results of scheduled runs of our implementation of this verification
algorithm are included in the information available on the DefElement website.}

\keywords{finite element methods}



\maketitle

\section{Introduction}\label{intro}

The finite element method has its origins in the 1940s, when methods involving the meshing of a domain began to be used by engineers \cite{courant,80years,25years}.
The name \emph{finite element method} was coined in a 1960 paper by Ray Clough \cite{clough1960}. Throughout the 1960s, significant progress was made
in analysing the finite element method.
The use of a triple to represent a finite element was introduced in 1975 in a set of lecture notes by Philippe Ciarlet \cite{CiarletLectureNotes},
before appearing in his popular 1978 book \cite{ciarlet}. This definition was first written with scalar-valued elements in mind, but it was used shortly after its introduction by
Pierre-Arnaud Raviart and Jean-Marie Thomas \cite{rt} and Jean-Claude N\'ed\'elec \cite{nedelec1} to define vector-valued H(div)-conforming and H(curl)-conforming elements.

Since its introduction, Ciarlet's triple---with some small notable adjustments---has become the widely accepted definition of a finite element, and it has been used to define a huge range of elements
designed for specific applications including
fluid dynamics \cite{guzman-neilan,rannacher-turek},
plate deformation \cite{pechstein-schoberl},
and solid mechanics \cite{serendipity}.
Details of Ciarlet's triple and an overview of different element types is presented in \cref{sec:definition}.

Often, only the lowest-order version of an element on a single cell type is introduced initially, with higher-order
versions and versions on other cell types presented in subsequent papers. This can make tracking down
appropriate elements that would be suitable for solving a particular problem a challenging task.
To simplify this task, we have created DefElement, an online encyclopedia of finite element definitions hosted
at \href{https://defelement.org}{defelement.org}\footnote{DefElement was originally hosted at \href{https://defelement.com}{defelement.com}.
This old URL now forwards to \href{https://defelement.org}{defelement.org}}.
The name DefElement is a shortening of \emph{def}initions of \emph{element}s, and partially
inspired by the use of the \texttt{def} keyword in Python \cite{python}. DefElement contains
the definitions of a large number of finite elements and useful information about each one.
Our encyclopedia was partly inspired by the Periodic Table of Finite Elements \cite{periodic-table}, which displays the elements in four common de Rham families,
and the Online Encyclopedia of Integer Sequences (OEIS) \cite{oeis}.
DefElement includes a page with information about each finite element.
As shown in \cref{fig:intro-website1} for the Raviart--Thomas element, this information includes:
\begin{itemize}
\item alternative names of the element;
\item details of the element's definition;
\item the polynomial and Lagrange subdegree and superdegree of the element, as we will define in \cref{sec:definition:degree};
\item the number of degrees of freedom (DOFs) that the element has in terms of the degree $k$;
\item details of any variants of the element, as we will define in \cref{sec:variants};
\item details of libraries in which the element is implemented.
Currently, the libraries Basix \cite{basix}, Bempp-cl \cite{bempp-cl}, FIAT \cite{FIAT}, ndelement \cite{ndelement}, and Symfem \cite{symfem} are included
\item some example low-degree elements, with functions and plots generated using Symfem.
These plots, and all other diagrams on DefElement, are licensed under a Creative Commons license, so that they can be reused
in papers and elsewhere as long as an attribution is included;
\item a list of references where the full definition of the element can be found.
\end{itemize}

\begin{figure}
\centering
\includegraphics[width=0.42\textwidth]{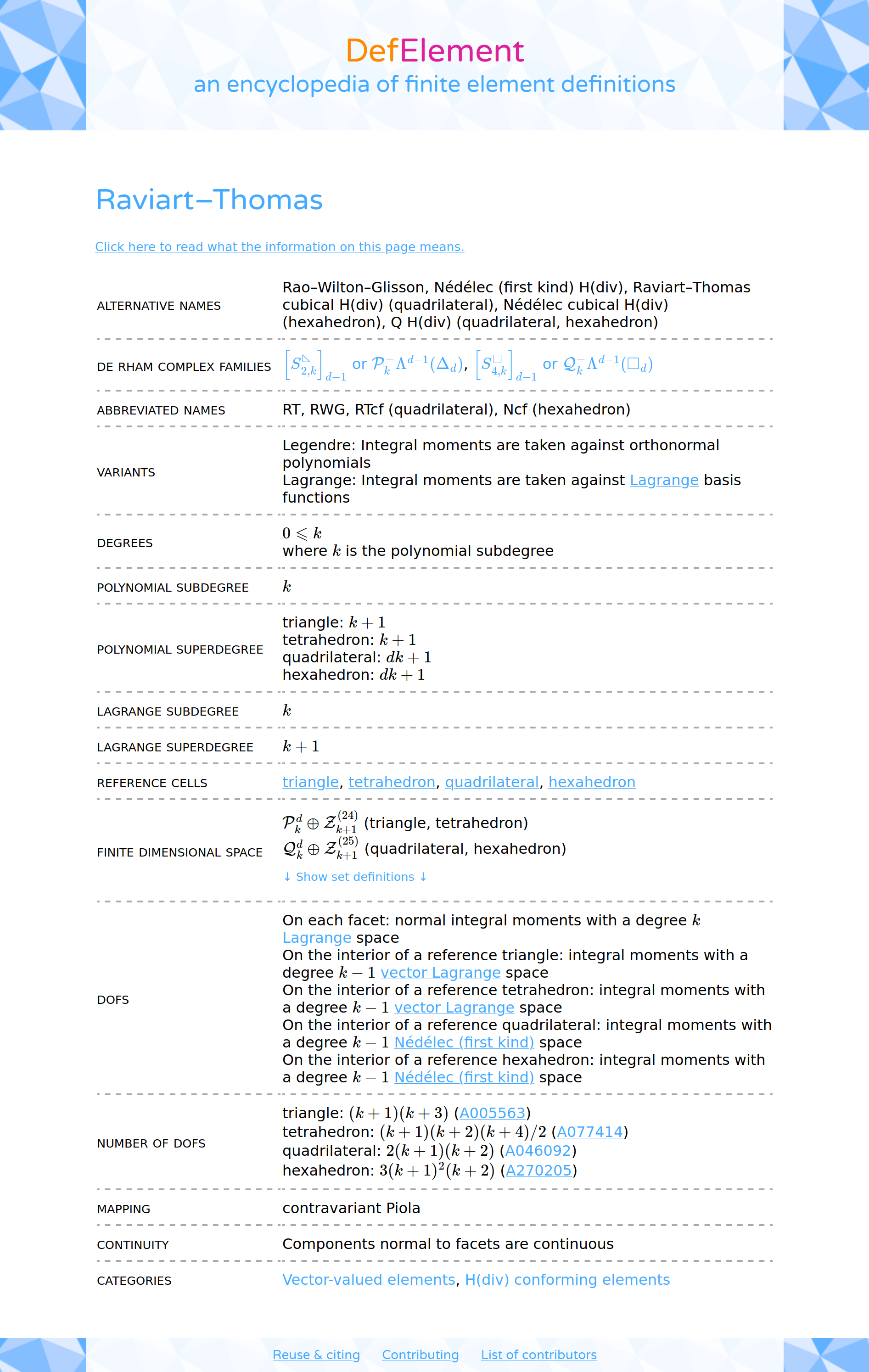}%
\hspace{1mm}%
\includegraphics[width=0.42\textwidth]{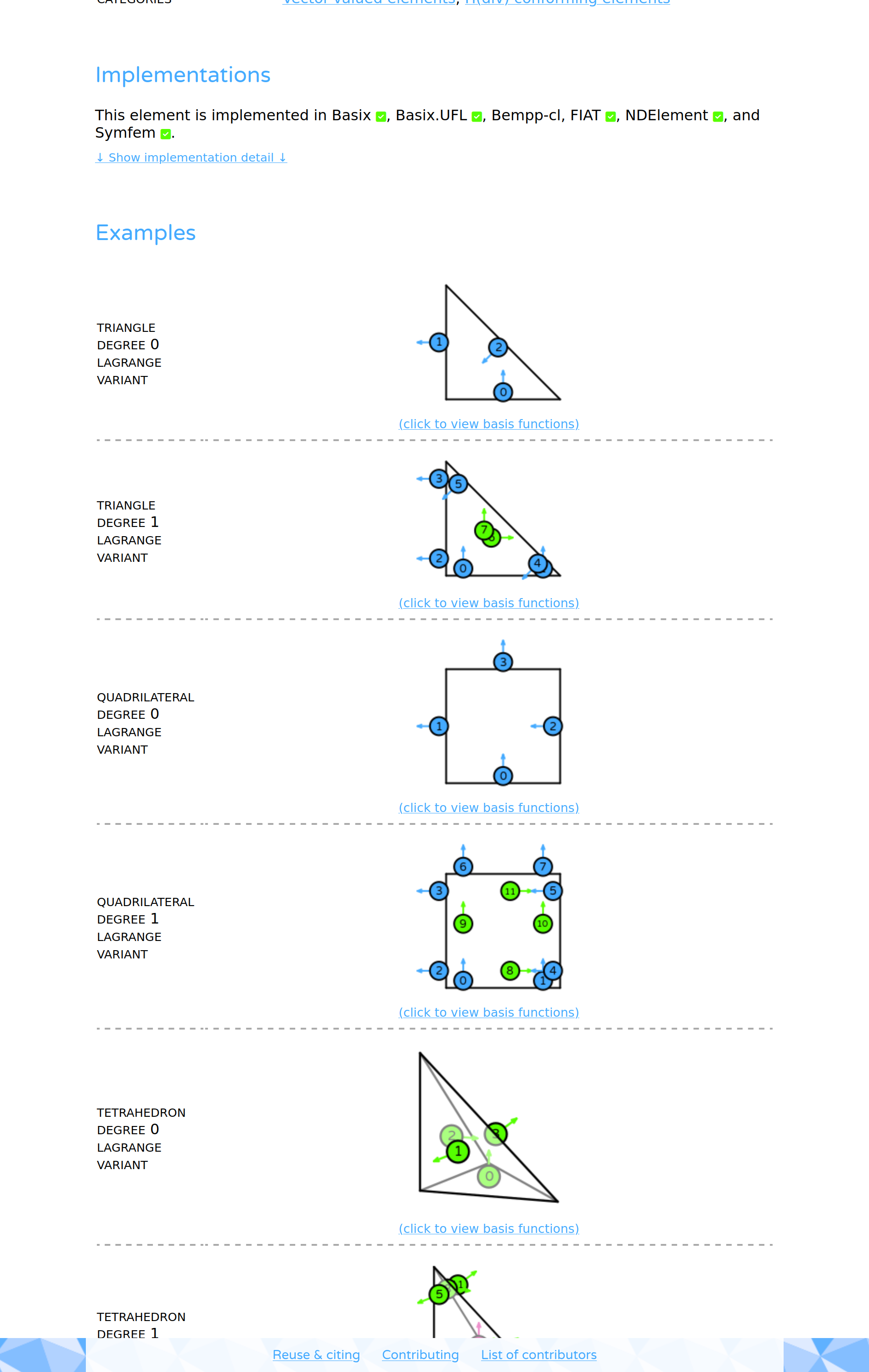}\\[2mm]
\includegraphics[width=0.42\textwidth]{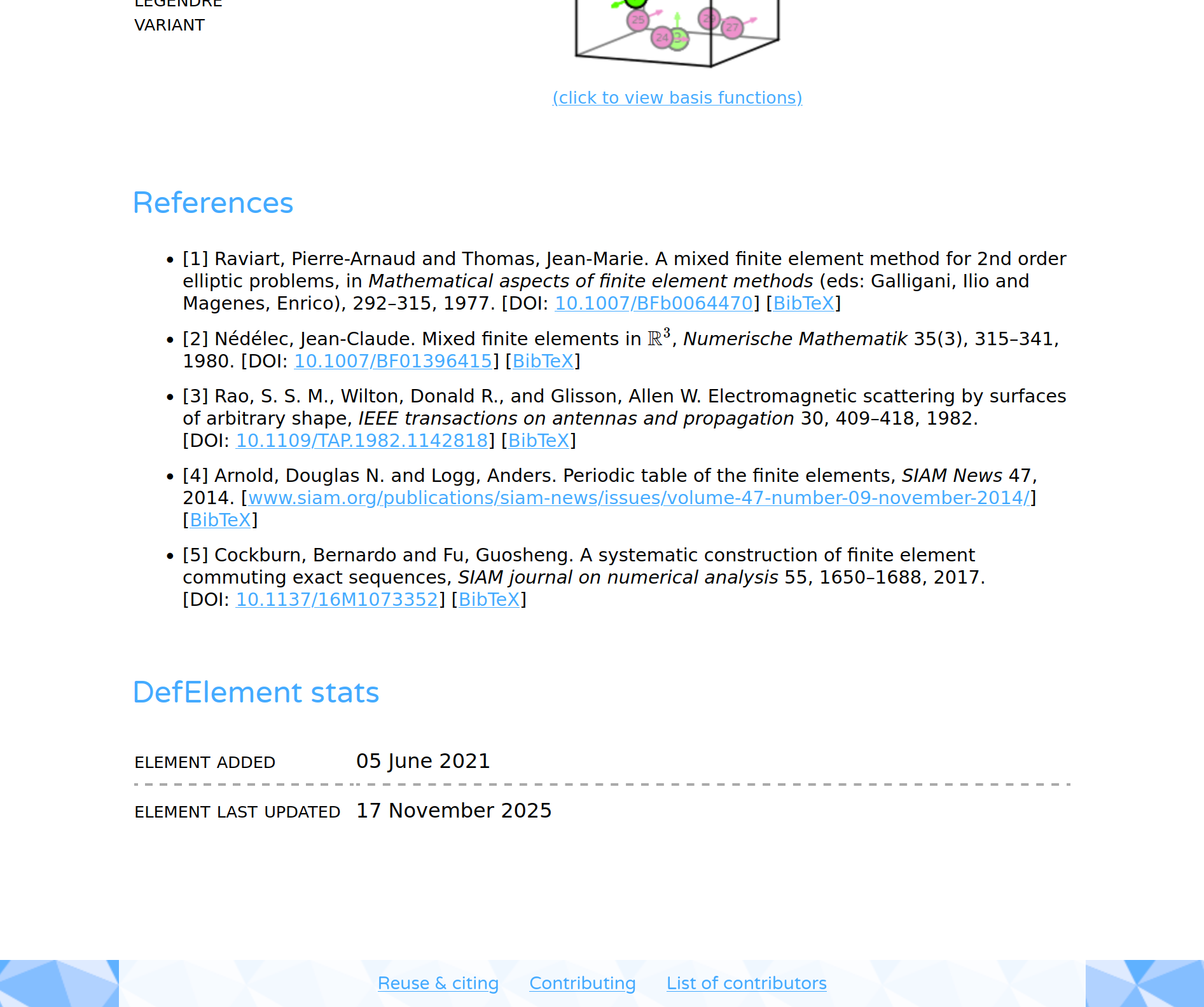}
\caption{Information about the Raviart--Thomas element, as shown on its DefElement page at \href{https://defelement.org/elements/raviart-thomas.html}{defelement.org/elements/raviart-thomas.html} (accessed 11 December 2025).
This information includes details of the element's definition (top left), implementations that include
the element and example low-degree elements (top right), and references to literature where full
details of the element can be found.
For Raviart--Thomas elements, the included references are \cite{rt,nedelec1,raviart-thomas-2,periodic-table,cockburn-fu}.}
\label{fig:intro-website1}
\end{figure}

Alongside the implementation detail, the DefElement page shows the status of DefElement's element verification---in \cref{fig:intro-website1},
this status is shown via a green tick or orange or red cross next to a library's name.
The aim of element verification is to confirm that the implementation of the element
in each library matches the definition given in the literature.
Element verification will be discussed in more detail in \cref{sec:verification}.

We encourage readers of this paper to consider contributing to DefElement by updating the information currently available
or adding new information and new elements.
In \cref{sec:editing}, we outline how contributions can be made to DefElement.
We finish with some concluding 
remarks in \cref{sec:conclusions}.

\section{Defining finite elements}\label{sec:definition}

In this section, we give an overview of the definition of finite elements. First, we define the names we use for
the sub-entities of a reference cell.
These definitions are summarised for
cells of topological dimensions 0 to 4 in
\cref{table:entities}, although these could trivially be extended to higher dimensional cells.

\begin{definition}[Sub-entities]\label{def:sub-entities}
The sub-entities of a polytope $\refel\subset\mathbb{R}^d$
of dimension 0, 1, 2 and 3 are called
vertices, edges, faces, and volumes (respectively).
The \emph{codimension} of a sub-entity is computed by subtracting the dimension of the sub-entity from the topological dimension $d$.
Sub-entities of codimension 0, 1, 2 and 3 are called \emph{the cell}, \emph{facets}, \emph{ridges} and \emph{peaks} (respectively).

For $k>3$, we do not introduce specific names for $k$-dimensional and $k$-codimensional sub-entities.
\end{definition}

\begin{table}
\renewcommand{\arraystretch}{1.2}
\begin{tabular}{c|ccccc}
\multirow{2}{20mm}{\centering Topological dimension}&
\multicolumn{4}{c}{Entities by dimension}\\
&0&1&2&3&4\\\hline
0 \footnotesize(a vertex)
&the cell&-&-&-&-\\
1 \footnotesize(an interval)
&vertices / facets&the cell&-&-&-\\
2 \footnotesize(a polygon)
&vertices / ridges&edges / facets&the cell&-&-\\
3 \footnotesize(a polyhedron)
&vertices / peaks&edges / ridges&faces / facets&the cell&-\\
4
&vertices&edges / peaks&faces / ridges&volumes / facets&the cell\\
\end{tabular}

\vspace{5mm}

\begin{tabular}{c|ccccc}
\multirow{2}{20mm}{\centering Topological dimension}&
\multicolumn{4}{c}{Entities by codimension}\\
&0&1&2&3&4\\\hline
0 \footnotesize(a vertex)
&the cell&-&-&-&-\\
1 \footnotesize(an interval)
&the cell&vertices / facets&-&-&-\\
2 \footnotesize(a polygon)
&the cell&edges / facets&vertices / ridges&-&-\\
3 \footnotesize(a polyhedron)
&the cell&faces / facets&edges / ridges&vertices / peaks&-\\
4
&the cell&volumes / facets&faces / ridges&edges / peaks&vertices\\
\end{tabular}
\vspace{2mm}
\caption{The entities of cells with topological dimensions 0 to 4.}
\label{table:entities}
\end{table}

In the most general terms, a finite element can be defined as follows.

\begin{definition}[Direct finite element]\label{def:direct}
A (direct) finite element is defined by the triple
$(\cell,\polyspace,\dualbasis)$, where
\begin{itemize}
  \item $\cell\subset\mathbb{R}^{d_g}$ is a $d$-dimensional cell in a mesh;
  \item $\polyspace$ is a finite dimensional space on $\cell$
  of dimension $n$, usually a space of polynomials;
  \item $\dualbasis :=\{l_0,\dots,l_{n-1}\}$ is a basis of the dual space
  $\dual{\polyspace} : =\{f:\polyspace\to\mathbb{R}\,|\,f\text{ is linear}\}$.
   Each functional $l_i$ is associated with a sub-entity of $\cell$. A functional
  $f\in\dualbasis$ associated with a sub-entity $\subentity\subseteq\cell$ only depends on a function's value
  restricted to $\subentity$, i.e.~if $\left.p\right|_{\subentity}=\left.q\right|_{\subentity}$ then $f(p)=f(q)$.
\end{itemize}

The basis functions $\{\phi_0,\dots,\phi_{n-1}\}$ of the finite element are
defined by
\[
  l_i(\phi_j)=\delta_{ij}:=\begin{cases}
  1&i=j,\\
  0&i\not=j.
\end{cases}
\]
\end{definition}

Using this definition, a different finite element can be defined for each physical cell in the mesh.
This allows for a wide range of approaches, such as $p$-refinement where different polynomial degrees
are used in different parts of the mesh.

When the element type and polynomial degree are the same for each cell, it is common to define
a single finite element on a reference cell and use a map to transform basis functions from the
reference cell to each physical cell in the mesh. A \emph{reference-mapped} finite element
that is defined for use in this case is defined as follows. In the literature, these are
often referred to as Ciarlet elements, due to their definition first appearing in
a set of lecture notes \cite{CiarletLectureNotes} and book \cite{ciarlet} by Philippe Ciarlet.

\begin{definition}[Reference-mapped finite element]\label{def:ciarlet}
A reference-mapped finite element is defined by the triple
$(\refel,\polyspace,\dualbasis)$, where
\begin{itemize}
  \item $\refel\subset\mathbb{R}^d$ is the reference cell, usually an interval, polygon or polyhedron;
  \item $\polyspace$ is a finite dimensional space on $\refel$
  of dimension $n$, usually a space of polynomials;
  \item $\dualbasis :=\{l_0,\dots,l_{n-1}\}$ is a basis of the dual space
  $\dual{\polyspace} : =\{f:\polyspace\to\mathbb{R}\,|\,f\text{ is linear}\}$.
  As in \cref{def:direct}, each functional $l_i$ is associated with a sub-entity of the reference cell $\refel$, with the value of any functional 
  $f\in\dualbasis$ associated with a sub-entity $\subentity\subseteq\refel$ only depending on a function's value
  restricted to $\subentity$.

\end{itemize}
The reference basis functions $\{\phi_0,\dots,\phi_{n-1}\}$ of the finite element are
defined by
\[
  l_i(\phi_j)=\delta_{ij}.
\]
\end{definition}

In \cref{def:sub-entities,def:ciarlet}, the value $d$ is called the topological dimension. This may differ from the geometric dimension $d_g$ if, for example,
a mesh of triangular cells is embedded in three-dimensional space. The \emph{value shape} of a finite element is defined
to be the shape of the range space of the polynomials in $\polyspace$. For example,
if $\polyspace$ is a set of vector-valued functions in $\{f\,|\,f:\refel\to\mathbb{R}^3\}$ then the value shape is $(3)$;
if $\polyspace$ is a set of matrix-valued functions in $\{f\,|\,f:\refel\to\mathbb{R}^{2\times2}\}$ then the value shape is $(2,2)$.
The \emph{value size} of a finite element is defined to be the product of all the entries of the value shape, and equal to 1
if the element is scalar-valued.

For reference-mapped elements, a geometric map $g$ is defined that maps points on the reference cell
to points on a physical cell and
a push forward map is defined that maps the basis functions on the
reference cell to functions on a physical cell
\cite{rognes:2009,mapping2,mapping3}. The push forward map for an element is
chosen so that it preserves specific properties of the element's basis functions.
The inverse map that maps function values from a physical cell to the reference cell is called the pull back map.
For some elements (such as serendipity \cite{serendipity}), better
convergence is achieved if the geometric map is affine, as defined as follows.

\begin{definition}[affine transformation]
Let $g:\refel\to\mathbb{R}^{d_g}$ be a geometric map.
$g$ is affine if there exist
a translation $g_1:\refel\to\mathbb{R}^d$
and a linear transformation $g_2:\mathbb{R}^d\to\mathbb{R}^{d_g}$ such that
$g=g_2\circ g_1$.
\end{definition}

The functionals $l_i \in \dualbasis$ are the (local) degrees-of-freedom (DOFs) of the finite
element (note that we enumerate functionals and basis functions starting at 0).
When a finite element function space is defined on a mesh, we associate a global DOF number
with each local DOF on each cell. To ensure that the mapped
space has the desired continuity properties, corresponding local DOFs on neighbouring cells
that are associated with a shared sub-entity are assigned the same global DOF.

We note that associating each functional with a sub-entity of the reference cell was not included
in Ciarlet's original definition of a finite element, but it is now common to include this as
this association is an important part of most finite element implementations, as it allows for
continuity between cells to be naturally implemented.
For continuity to be effectively imposed in this way, it is common to make some additional assumptions about the finite element.
Before stating these assumptions, we first define equivalent sets of functionals.

\begin{definition}[equivalent sets of functionals]\label{def:equivalent}
Let $\dualbasis=\{l_0,\dots,l_{a-1}\}$ and $\tilde\dualbasis=\{\tilde{l}_0,\dots,\tilde{l}_{b-1}\}$ be two sets of functionals associated with entities $\subentity$ and $\tilde{\subentity}$.
Let $\{\phi_0,\dots,\phi_{a-1}\}$ and $\{\tilde{\phi}_0,\dots,\tilde{\phi}_{b-1}\}$ be the finite element basis functions corresponding to the functionals in $\dualbasis$ and $\tilde{\dualbasis}$.
The sets of functionals $\dualbasis$ and $\tilde\dualbasis$ are equivalent if:
\begin{itemize}
\item they have the same size (i.e.~$a=b$);
\item there exists a push forward map $\mathcal{F}$ such that for all $i=0,\dots,a-1$,
$$\tilde{l}_i(\mathcal{F}\phi_j)=\delta_{ij}.$$
\end{itemize}
\end{definition}

For sets of functionals associated with the same sub-entity, the following lemma will be useful.

\begin{lemma}\label{lemma:equivalent}
Let $\dualbasis=\{l_0,\dots,l_{a-1}\}$ and $\tilde\dualbasis=\{\tilde{l}_0,\dots,\tilde{l}_{a-1}\}$ be two sets of functionals associated with entity $\subentity$.
If \(\operatorname{span}\dualbasis=\operatorname{span}\tilde{\dualbasis}\), then $\dualbasis$ and $\tilde{\dualbasis}$ are equivalent sets of functionals.
\end{lemma}
\begin{proof}
As \(\operatorname{span}\dualbasis=\operatorname{span}\tilde{\dualbasis}\), there exists a non-singular matrix $\mat{A}\in\mathbb{R}^{a\times a}$ such that
\begin{equation}\label{eq:defA}
\begin{pmatrix}\tilde{l}_0\\\vdots\\\tilde{l}_{a-1}\end{pmatrix}
=\mat{A}
\begin{pmatrix}l_0\\\vdots\\l_{a-1}\end{pmatrix}.
\end{equation}
Note that as $\dualbasis$ and $\tilde\dualbasis$ are not necessarily orthogonal bases,
$\mat{A}$ is not necessarily an orthogonal matrix.

Let $\{\phi_0,\dots,\phi_{a-1}\}$ be the finite element basis functions corresponding to the functionals in $\dualbasis$.
The aim is to show that the second bullet point in \cref{def:equivalent} is true: we will see that
this holds if we define $\mathcal{F}$ as
\[
\begin{pmatrix}\mathcal{F}\phi_0\\\vdots\\\mathcal{F}\phi_{a-1}\end{pmatrix}
=\mat{A}^{-\mathsf{T}}
\begin{pmatrix}\phi_0\\\vdots\\\phi_{a-1}\end{pmatrix}.
\]
As the functionals $\tilde{l}_i$ are linear, we can write
\begin{equation}
\begin{pmatrix}
\tilde{l}_0(\mathcal{F}\phi_0)&\dots&\tilde{l}_{a-1}(\mathcal{F}\phi_0)\\
\vdots&\ddots&\vdots\\
\tilde{l}_0(\mathcal{F}\phi_{a-1})&\dots&\tilde{l}_{a-1}(\mathcal{F}\phi_{a-1})\\
\end{pmatrix}
=
\mat{A}^{-\mathsf{T}}
\begin{pmatrix}
\tilde{l}_0(\phi_0)&\dots&\tilde{l}_{a-1}(\phi_0)\\
\vdots&\ddots&\vdots\\
\tilde{l}_0(\phi_{a-1})&\dots&\tilde{l}_{a-1}(\phi_{a-1})\\
\end{pmatrix}.\label{eq:first_to_combine}
\end{equation}
Using \cref{eq:defA}, we see that
\begin{equation}
\begin{pmatrix}
\tilde{l}_0(\phi_0)&\dots&\tilde{l}_{a-1}(\phi_0)\\
\vdots&\ddots&\vdots\\
\tilde{l}_0(\phi_{a-1})&\dots&\tilde{l}_{a-1}(\phi_{a-1})\\
\end{pmatrix}
=
\begin{pmatrix}
l_0(\phi_0)&\dots&l_{a-1}(\phi_0)\\
\vdots&\ddots&\vdots\\
l_0(\phi_{a-1})&\dots&l_{a-1}(\phi_{a-1})\\
\end{pmatrix}
\mat{A}^\mathsf{T}.\label{eq:second_to_combine}
\end{equation}
Combining \cref{eq:first_to_combine,eq:second_to_combine}, we see that
\begin{align*}
\begin{pmatrix}
\tilde{l}_0(\mathcal{F}\phi_0)&\dots&\tilde{l}_{a-1}(\mathcal{F}\phi_0)\\
\vdots&\ddots&\vdots\\
\tilde{l}_0(\mathcal{F}\phi_{a-1})&\dots&\tilde{l}_{a-1}(\mathcal{F}\phi_{a-1})\\
\end{pmatrix}
&=
\mat{A}^{-\mathsf{T}}
\begin{pmatrix}
l_0(\phi_0)&\dots&l_{a-1}(\phi_0)\\
\vdots&\ddots&\vdots\\
l_0(\phi_{a-1})&\dots&l_{a-1}(\phi_{a-1})\\
\end{pmatrix}
\mat{A}^\mathsf{T}
\\&=
\mat{A}^{-\mathsf{T}}
\mat{I}
\mat{A}^\mathsf{T}
\\&=
\mat{A}^{-\mathsf{T}}
\mat{A}^\mathsf{T}
\\&=
\mat{I}.
\end{align*}
\end{proof}

We now present the additional assumptions about reference-mapped finite elements that are commonly made to allow the elements
to be mapped from the reference cell to every physical cell in an unstructured mesh.
First, the following assumption is often made:

\begin{assumption}\label{a:equivalent}
Each sub-entity of the same type (e.g.~each sub-entity that is a triangle) has an equivalent set of functionals associated with it.
\end{assumption}

This assumption allows neighbouring cells to be attached by any sub-entity of the same type,
as each sub-entity of the given type will have an equivalent set of DOFs.
\Cref{a:equivalent} is true of the vast majority of finite elements, but it is not true of some such as
transition elements, where different polynomial degrees are used on different sub-entities,
and Fortin--Soulie elements \cite{fortin-soulie}, where there are two DOFs associates with two edges
of the reference triangle but only one DOF associated with the third edge.

Note that in some finite element libraries---in particular those that use $p$- or $hp$-refinement
or hierarchical basis functions---functionals can be associated with sub-entities where no corresponding
functional is associated in a neighbouring cell. Where the functionals in question are point
evaluations, these are referred to as \emph{hanging nodes}. In such cases, continuity must be
enforced in a less trivial manner, and \cref{a:equivalent} cannot be made.

A second common assumption is as follows:

\begin{assumption}\label{a:orientation}
Let $\subentity$ be a sub-entity of $\refel$, let $g:E\to E$ be an affine bijection, and
$\mathcal{F}_g$ be the push forward corresponding to $g$.
If $\phi_a,\dots,\phi_b$ are the basis functions associated with the sub-entity $\subentity$, then
$$
\operatorname{span}\left\{\mathcal{F}_g(\left.\phi_a\right|_{\subentity}),\dots,\mathcal{F}_g(\left.\phi_b\right|_{\subentity})\right\} =
\operatorname{span}\left\{\left.\phi_a\right|_{\subentity},\dots,\left.\phi_b\right|_{\subentity}\right\},
$$
where $|_{\subentity}$ denotes the restriction of a function to the sub-entity $\subentity$.
\end{assumption}

This assumption allows neighbouring cells to be attached in any orientation, as 
any map that corresponds to a rotation or reflection of a sub-entity is affine and a bijection \cite[lemma 20.6, exercise 20.1]{ern-guermond},
and so \cref{a:orientation} tells us that the space that needs to be continuous between cells
is not affected by rotating or reflecting a sub-entity. Note, however, that the differences in the orientation of neighbouring cells
must be taken into account in implementations of finite elements on unordered meshes \cite{2022-dofs}.

In software implementations, it is common to compute basis functions by first computing the dual matrix \cite[section 2.2]{FIAT}
$$
\mat{D} = \begin{pmatrix}
l_0(p_0)&\dots&l_{n-1}(p_0)\\
\vdots&\ddots&\vdots\\
l_0(p_{n-1})&\dots&l_{n-1}(p_{n-1})
\end{pmatrix},
$$
where $\{p_0,\dots,p_{n-1}\}$ is any basis of $\polyspace$ (often chosen to be an orthogonal basis).
Multiplying by the vector of basis functions, we see that
\[
\mat{D}\begin{pmatrix}
\phi_0\\\vdots\\\phi_{n-1}
\end{pmatrix}
=
\begin{pmatrix}
\displaystyle
\sum_{i=0}^{n-1}l_i(p_0)\phi_i
\\\vdots\\
\displaystyle
\sum_{i=0}^{n-1}l_i(p_{n-1})\phi_i
\end{pmatrix}
=
\begin{pmatrix}
p_0\\\vdots\\p_{n-1}
\end{pmatrix},
\]
and so the basis functions can be computed via
\begin{equation}
\begin{pmatrix}
\phi_0\\\vdots\\\phi_{n-1}
\end{pmatrix}
=
\mat{D}\inv
\begin{pmatrix}
p_0\\\vdots\\p_{n-1}
\end{pmatrix}.\label{eq:dual-inv}
\end{equation}
Note that if the functions are not being computed symbolically, the values of $p_0(\vec{x}),\dots,p_{n-1}(\vec{x})$ at
a given point $\vec{x}\in\refel$ can be used on the right-hand side of \cref{eq:dual-inv} to compute the values of the basis functions at $\vec{x}$.

Note that if the basis $\{p_0,\dots,p_{n-1}\}$ was chosen to be the set of monomials
and the functionals $l_0,\dots,l_{n-1}$ were all point evaluations, then $\mat{D}$
would be a Vandermonde matrix with an ill-conditioned inverse \cite{vandermonde}.
The matrix $\mat{D}$ for a general finite element is sometimes called the generalised
Vandermonde matrix \cite[chapter 5]{ern-guermond}.

We define conforming and nonconforming finite elements as follows.

\begin{definition}[conforming and nonconforming]
Let $V$ be a function space and let $M_h$ be a mesh of the domain of $V$.
A finite element $(\refel, \polyspace, \dualbasis)$ is $V$-conforming if the basis functions of the element on $M_h$ are in $V$.
A finite element $(\refel, \polyspace, \dualbasis)$ is called $V$-nonconforming if it is not $V$-conforming.
\end{definition}

\subsection{Element types}\label{sec:definition:types}
In this section, we present a brief overview of commonly used types of finite element.
\subsubsection{Scalar-valued elements}
Scalar-valued elements are most commonly used to discretise the $H^1$ Sobolev space, defined
for $\Omega\subset\mathbb{R}^d$ by
\[
H^1(\Omega):=\left\{f\in L^2(\Omega)\,\middle|\,
\frac{\partial f}{\partial x_0},\dots,
\frac{\partial f}{\partial x_{d-1}}
\in L^2(\Omega)\right\},
\]
where \(x_0=x\), \(x_1=y\), \(x_2=z\), and
\[
L^2(\Omega):=H^0(\Omega):=\left\{f:\Omega\to\mathbb{R}\,\middle|\,\int_\Omega f^2 <\infty\right\}.
\]

The scalar-valued Lagrange element is probably the most
widely used element and is the first element introduced in most finite element courses. The polynomial space $\polyspace$ of a Lagrange element
on a simplex cell is a complete polynomial space (the polynomial spaces for other cells are later defined in \cref{def:polyspaces}),
and its dual basis $\dualbasis$ is a set of point evaluations.
On quadrilateral and hexahedral cells, serendipity \cite{serendipity} elements are a second widely-used $H^1$-conforming finite element: their polynomial spaces
$\polyspace$ are a smaller set of polynomials than a Lagrange element and their dual bases $\dualbasis$ contain point evaluations at vertices and integral moments on
other sub-entities. Serendipity elements achieve the same accuracy as Lagrange elements on affine cells, but their accuracy degrades on non-affine cells \cite{abf-approx}.

By including derivatives in the definitions of functionals, we can enforce continuity on some derivatives of an element and create elements for
discretising $H^2$ and higher-order spaces such as the Argyris \cite{argyris}, Bell \cite{bell}, Hermite \cite{hermite}, Morley \cite{morley}, Morley--Wang--Xu \cite{morley-wang-xu} and Wu--Xu elements \cite{wu-xu}.
Note that in some older finite element literature (e.g.~\cite{ciarlet}), Hermite was used as the name for \emph{all} finite elements with at least one derivative
included in their functionals whereas it is now widely used as the name for one specific element.

By associating all the DOFs of an element with the interior of the cell, we can create an element that is discontinuous between cells that can be used to
discretise $L^2$.
The discontinuous Galerkin (DG) finite element method \cite{arnold-dg,Reed1973Triangular} is often based on use of the discontinuous Lagrange element.
A number of elements can also be created that are $H^1$-nonconforming---they are neither fully discontinuous nor fully $C^0$ continuous. The most widely used of these is probably the
Crouzeix--Raviart \cite{crouzeix-raviart} element alongside its extensions to higher polynomial degrees and quadrilaterals \cite{crouzeix-falk,fortin-soulie,rannacher-turek}.

$H^1$-conforming scalar-valued elements typically use the identity push-forward map, also known as the push-forward by the geometric transformation. This is defined, for a scalar-valued function $\phi$, by
$$\left(\mapping^\text{id}\phi\right)(\vec{x})
:=\phi(\geometrymap\inv(\vec{x})),$$
where $\geometrymap:\refel\to\mathbb{R}^{d_g}$ is the geometric map from the reference cell $\refel\subset\mathbb{R}^d$ to a physical cell.
The term \(\geometrymap\inv(\vec{x})\) is the point on the reference cell corresponding to the point \(\vec{x}\), so this mapping
maps a value of the function on the reference cell to the same value at the corresponding point.

$L^2$-conforming elements---such as discontinuous Lagrange elements---use the $L^2$ Piola push-forward that is defined, for a scalar-valued function $\phi$, by
$$\left(\mapping^\text{L2}\phi\right)(\vec{x})
:=\frac1{\det \jacobian}\phi(\geometrymap\inv(\vec{x})),$$
where $\jacobian$ is the Jacobian of the transformation $\geometrymap$.
When $\geometrymap$ is affine, $\det\jacobian$ will be constant and the identity push forward
can be used (and is often used) instead of the $L^2$ Piola push forward.
However, if $\geometrymap$ is not affine, then in some applications it is necessary to use the $L^2$ Piola map
so that the basis functions maintain properties of the de Rham complex (see \cref{sec:definition:complexes}).

The mapping of elements involving derivatives in their functionals from the reference cell to a physical cell is a more
involved process \cite{mapping3}, whose detailed discussion is beyond the scope of this paper.

\subsubsection{Vector-valued elements}
Many problems, including those in electromagnetics \cite{monk-book} and fluids \cite{fluids-book},
involve vectors. For such problems, vector-valued finite elements can be used. Vector-valued
elements can be created by using a separate scalar-valued element for each component of the vector:
such elements are typically used to represent the geometry of a mesh, with each component of the
position of a cell typically represented by a Lagrange element. Using higher-degree elements in
the representation of geometry allows curved cells to be represented.

Vector-valued elements built from multiple copies of a scalar-valued element copy the continuity
of the scalar element---for most scalar-valued elements, this leads to vector functions whose
components are all continuous across boundaries between cells. In some cases---for example, when
representing geometry---this $C^0$ continuity is desired, but in other cases, the discontinuity of
some components is more appropriate.

The spaces $H(\sdiv)$ and $H(\vcurl)$
\cite[chapter 4]{ern-guermond}
are defined for $\Omega\subset\mathbb{R}^d$ by
\begin{align*}
H(\sdiv, \Omega) &:=\left\{\vec{f}\in \left(L^2(\Omega)\right)^d\,\middle|\,\sdiv(\vec{f})\in L^2(\Omega)\right\},\\
H(\scurl, \Omega) &:=\left\{\vec{f}\in \left(L^2(\Omega)\right)^d\,\middle|\,\scurl(\vec{f})\in\tilde{L}\right\},
\end{align*}
where $$\tilde{L}:=\begin{cases}\left(L^2(\Omega)\right)^3&\text{if }d=3,\\L^2(\Omega)&\text{if }d=2,\end{cases}$$
and for $d=2$, $\scurl$ is defined by
$$\scurl\left(\begin{pmatrix}f_0\\f_1\end{pmatrix}\right)
=
\frac{\partial f_1}{\partial x} -\frac{\partial f_0}{\partial y}.$$
When discretising $H(\sdiv)$, we only require that components of the vector
normal to each facet are continuous between cells.
Raviart--Thomas \cite{rt}
and Brezzi--Douglas--Marini \cite{bdm} elements are the most widely used $H(\sdiv)$-conforming elements. To enforce
the appropriate continuity, the DOFs associated with each facet are integral moments against vector functions
normal to the faces.
When discretising $H(\vcurl)$, we only require that the components of the vector
tangent to edges and faces are continuous.
N\'edelec first kind \cite{nedelec1} and second kind \cite{nedelec2} elements
are the most widely used $H(\vcurl)$-conforming elements. To enforce the appropriate continuity, the DOFs associated with each
edge and face are integral moments against vector functions tangential to the sub-entity.

$H(\sdiv)$-conforming and $H(\vcurl)$-conforming elements
use the contravariant Piola ($\mapping^{\sdiv}$) and covariant Piola ($\mapping^{\vcurl}$) push-forward
maps (respectively). These are defined by
\begin{align*}
\left(\mapping^{\sdiv}\vec{\phi}\right)(\vec{x})
&:=\frac1{\det\jacobian}\jacobian\vec{\phi}(\geometrymap\inv(\vec{x})),\\
\left(\mapping^{\vcurl}\vec{\phi}\right)(\vec{x})
&:=\jacobian\invtranspose\vec{\phi}(\geometrymap\inv(\vec{x})).
\end{align*}
These maps preserve the continuity of normal and tangential components of $H(\sdiv)$-conforming and
$H(\vcurl)$-conforming elements (respectively).

\subsubsection{Matrix-valued elements}
Some problems involve matrix-valued unknowns, such as the strain and stress in elasticity \cite{regge-1}.
For such problems, matrix-valued finite elements can be defined.
As in the vector case, these can be created be using a scalar-valued element for each
entry of the matrix. Depending on the application, symmetric or non-symmetric matrices may be
desired for the unknown: both cases can easily be created from scalar-valued element by either
using a seperate element for each entry, or by using fewer elements and setting the appropriate
matrix entries to be equal.

As in the vector-valued case, for many applications we desire elements with weaker continuity than the
$C^0$ continuity usually obtained from the scalar element approach.
For any matrix-valued function $\mat{F}$, we define
$\sdiv(\mat{F})$ to be the vector containing the results of applying $\sdiv$ to each row of $\mat{F}$.
If $d=2$, we define $\scurl(\mat{F})$ to be the vector containing the results of applying $\scurl$ to each row of $\mat{F}$;
if $d=3$, we define $\scurl(\mat{F})$ to be the matrix where the $i$th column is the result of applying $\scurl$ to the $i$th row of $\mat{F}$.
We can now define the appropriate
Sobolev spaces for matrix-valued elements:
\newcommand{\curlcurllabel}{2}
\newcommand{\curldivlabel}{1}
$H(\sdiv\,\sdiv)$,
$H(\vcurl\,\vcurl)$ and
$H(\vcurl\,\sdiv)$ \cite{formulae-and-transformations,gopalakrishnan-lederer-schoberl}. These are defined for $\Omega\subset\mathbb{R}^d$ by
\begin{align*}
H(\sdiv\,\sdiv, \Omega)&:=\left\{
    \mat{F}\in\left(L^2(\Omega)\right)^{d\times d}
    \,\middle|\,
    \mat{F}\transpose=\mat{F}
    \text{ and }
    \sdiv(\sdiv(\mat{F}))\in\left(H^1_0(\Omega)\right)^*
\right\},\\
H(\vcurl\,\vcurl, \Omega)&:=\left\{
    \mat{F}\in\left(L^2(\Omega)\right)^{d\times d}
    \,\middle|\,
    \mat{F}\transpose=\mat{F}
    \text{ and }
    \scurl(\scurl(\mat{F}))\in\left(\tilde{H}_{\curlcurllabel}\right)^*
\right\},\\
H(\vcurl\,\sdiv, \Omega)&:=\left\{
    \mat{F}\in\left(L^2(\Omega)\right)^{d\times d}
    \,\middle|\,
    \scurl(\sdiv(\mat{F}))\in\left(\tilde{H}_{\curldivlabel}\right)^*
\right\},
\end{align*}
where
\begin{align*}
H^1_0(\Omega)&:=\left\{f\in H^1(\Omega)\,\middle|\,f=0\text{ on }\partial\Omega\right\},\\
\tilde{H}_{\curldivlabel}&:=\begin{cases}\left(H^1_0(\Omega)\right)^3&\text{if }d=3,\\H^1_0(\Omega)&\text{if }d=2,\end{cases}\\
\tilde{H}_{\curlcurllabel}&:=\begin{cases}\left(H^1_0(\Omega)\right)^{3\times3}&\text{if }d=3,\\H^1_0(\Omega)&\text{if }d=2,\end{cases}
\end{align*}
and $V^*$ denotes the dual of a space $V$.

Note that elements in $H(\sdiv\,\sdiv)$ and $H(\vcurl\,\vcurl)$ are symmetric matrices, while elements
in $H(\vcurl\,\sdiv)$ are non-symmetric.
When discretising $H(\sdiv\,\sdiv)$,
we require that
the normal-normal component
$$\vec{n}\transpose\mat{\Phi}\vec{n}$$
is continuous across each facet,
where $\vec{n}$ is normal to the facet.
Regge elements \cite{regge-0,regge-1}
are the most widely used
$H(\sdiv\,\sdiv)$-conforming elements.
When discretising $H(\vcurl\,\vcurl)$,
we require that
the tangent-tangent component
$$\vec{t}\transpose\mat{\Phi}\vec{t}$$
is continuous across each sub-entity,
where $\vec{t}$ is tangential to the sub-entity.
Hellan--Herrmann--Johnson elements \cite{hhj}
are the most widely used
$H(\vcurl\,\vcurl)$-conforming elements.
When discretising $H(\vcurl\,\sdiv)$,
we require that
the tangent-normal component
$$\vec{t}\transpose\mat{\Phi}\vec{n}$$
is continuous across each facet.
Sets of $H(\vcurl\,\sdiv)$-conforming elements are defined in \cite{hu-lin,hu-lin-zhang}.
Gopalakrishnan--Lederer--Sch\"oberl elements \cite{gopalakrishnan-lederer-schoberl}\cite[section 5.3]{lederer-schoberl}
can be used to discretise
$H(\vcurl\,\sdiv)$, although they are slightly nonconforming in $H(\vcurl\,\sdiv)$.

$H(\sdiv\,\sdiv)$-conforming,
$H(\vcurl\,\vcurl)$-conforming and
$H(\vcurl\,\sdiv)$-conforming elements use the
double contravariant Piola ($\mapping^{\sdiv\,\sdiv}$),
double covariant Piola ($\mapping^{\vcurl\,\vcurl}$)
and
covariant-contravariant Piola ($\mapping^{\vcurl\,\sdiv}$)
push-forward maps (respectively). These are defined by
\begin{align*}
\left(\mapping^{\sdiv\,\sdiv}\mat{\Phi}\right)(\boldsymbol{x})&:=\frac1{\left(\det\jacobian\right)^2}\jacobian\mat{\Phi}(\geometrymap\inv(\boldsymbol{x}))\jacobian\transpose,\\
\left(\mapping^{\vcurl\,\vcurl}\mat{\Phi}\right)(\boldsymbol{x})&:=\jacobian\invtranspose\mat{\Phi}(\geometrymap\inv(\boldsymbol{x}))\jacobian\inv,\\
\left(\mapping^{\vcurl\,\sdiv}\mat{\Phi}\right)(\boldsymbol{x})&:=\frac1{\left(\det\jacobian\right)}\jacobian\invtranspose\mat{\Phi}(\geometrymap\inv(\boldsymbol{x}))\jacobian\transpose.
\end{align*}
These maps preserve the appropriate continuity for each element type.

\subsubsection{Macroelements}
The elements listed so far in this section are all defined using polynomials on each cell.
There are, however, many desired properties of elements that can be achieved more efficiently
by using macroelements, particularly when enforcing higher-order continuity and divergence-free constraints.
A macroelement is defined by subdividing the cell $\refel$ into a number of
sub-cells, and then defining $\polyspace$ to be polynomial on each subcell with some continuity weaker
than $C^\infty$ between subcells \cite{fiat-macro}. There are a large number of ways of splitting a cell
into sub-cells that used when defining macroelements. A selection of these are shown in \cref{fig:subcells}.

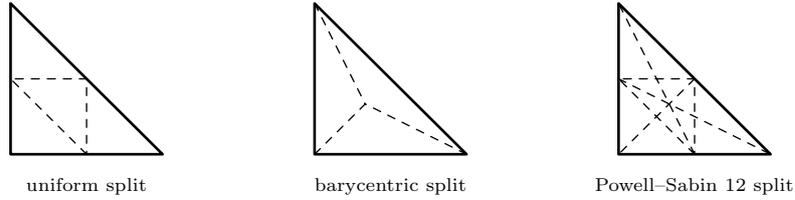
\begin{figure}
\centering
\footnotesize
\begin{tikzpicture}[line width=1pt,line cap=round,line join=round,scale=2]
\begin{scope}
\draw[dashed,line width=0.5pt] (0.5,0) -- (0,0.5) -- (0.5,0.5) -- cycle;
\draw (0,0) -- (1,0) -- (0,1) -- cycle;
\node[anchor=north] at (0.5,-0.1) {uniform split};
\end{scope}
\begin{scope}[shift={(2,0)}]
\draw[dashed,line width=0.5pt] (0,0) -- (0.3333,0.3333) (0,1) -- (0.3333,0.3333) (1,0) -- (0.3333,0.3333);
\draw (0,0) -- (1,0) -- (0,1) -- cycle;
\node[anchor=north] at (0.5,-0.1) {barycentric split};
\end{scope}
\begin{scope}[shift={(4,0)}]
\draw[dashed,line width=0.5pt] (0.5,0) -- (0,1) (0,0.5) -- (1,0) (0.5,0.5) -- (0,0) (0.5,0) -- (0,0.5) -- (0.5,0.5) -- cycle;
\draw (0,0) -- (1,0) -- (0,1) -- cycle;
\node[anchor=north] at (0.5,-0.1) {Powell--Sabin 12 split};
\end{scope}
\end{tikzpicture}
\caption{Three commonly used methods for splitting a triangular cell into sub-cells when defining macroelements.
The uniform split subdivides the triangle into four congruent sub-triangles and is used by P1-iso-P2 elements \cite{p1-iso-p2};
the barycentric split subdivides the triangle into three sub-triangles by adding a vertex at the barycentre and is used by Alfeld--Sorokina elements \cite{alfeld-sorokina}, Guzm\'an--Neilan elements \cite{guzman-neilan} and Hsieh--Clough--Tocher elements \cite{hct-0,hct-1,hct-2} among others;
the Powell--Sabin 12 split subdivides the triangle into 12 sub-triangules using the barycentre and the midpoints of each side and is used by some Powell--Sabin elements \cite{powell-sabin}.
}
\label{fig:subcells}
\end{figure}

Popular macroelements include
the P1-iso-P2 element \cite{p1-iso-p2}, which achieves some of the desired properties of a quadratic polynomial element while being linear on each sub-cell;
Hsieh--Clough--Tocher elements \cite{hct-0,hct-1,hct-2}, which achieve $C^1$ continuity using fewer degrees-of-freedom than other $C^1$ elements;
and Guzm\'an--Neilan elements \cite{guzman-neilan}, which are divergence-free at a lower polynomial degree than is otherwise possible.

\subsubsection{Dual elements}
In boundary element applications, elements defined on the barycentric dual grid are commonly used when constructing
operator preconditioners \cite{bc}. The barycentric dual grid is constructed from a triangular cell by
(i) barycentrically refining each cell (by connecting each vertex with the midpoint of its opposite edge)
then (ii) forming all the sub-triangles adjacent to a vertex in the unrefined
grid into a single cell. This process is shown for a small example grid in \cref{fig:barycentric-refinement}.

Dual elements can be treated as direct finite elements (\cref{def:direct}), as they are defined on
physical cells rather than being mapped from a reference cell. It is common, however, to define
each dual basis function as a linear combination of basis functions on the barycentrically refined
grid, where these barycentric basis functions are defined via a reference-mapped element.

\begin{figure}
\centering
\begin{tikzpicture}[line width=1pt,line cap=round,line join=round]
\draw[->] (2.5,1.2) -- +(0.5,0);
\draw[->] (5.8,1.2) -- +(0.5,0);
\draw[->] (9.1,1.2) -- +(0.5,0);
\node[anchor=south] at (2.75,1.2) {(i)};
\node[anchor=south] at (9.35,1.2) {(ii)};

\draw (0, 0) -- (0, 1.2) (0, 0) -- (1, 1) (0, 1.2) -- (1, 1) (0, 0) -- (1.3, -0.2) (1, 1) -- (1.3, -0.2) (1, 1) -- (2, 1) (2, 1) -- (1.3, -0.2) (1, 1) -- (2, 2) (2, 1) -- (2, 2) (1, 1) -- (0.8, 2.2) (2, 2) -- (0.8, 2.2) (2, 2) -- (1.4, 3) (1.4, 3) -- (0.8, 2.2) (1.4, 3) -- (-0.1, 2.5) (-0.1, 2.5) -- (0.8, 2.2) (0, 1.2) -- (-0.1, 2.5) (0, 1.2) -- (0.8, 2.2);
\fill[cyan] (1.3,-0.2) circle (2pt) (2,2) circle (2pt) (1,1) circle (2pt) (0,1.2) circle (2pt) (1.4,3) circle (2pt) (0,0) circle (2pt) (2,1) circle (2pt) (0.8,2.2) circle (2pt) (-0.1,2.5) circle (2pt);
\begin{scope}[shift={(3.3,0)}]
\draw (0, 0) -- (0, 1.2) (0, 0) -- (1, 1) (0, 1.2) -- (1, 1) (0, 0) -- (1.3, -0.2) (1, 1) -- (1.3, -0.2) (1, 1) -- (2, 1) (2, 1) -- (1.3, -0.2) (1, 1) -- (2, 2) (2, 1) -- (2, 2) (1, 1) -- (0.8, 2.2) (2, 2) -- (0.8, 2.2) (2, 2) -- (1.4, 3) (1.4, 3) -- (0.8, 2.2) (1.4, 3) -- (-0.1, 2.5) (-0.1, 2.5) -- (0.8, 2.2) (0, 1.2) -- (-0.1, 2.5) (0, 1.2) -- (0.8, 2.2);
\draw[dashed,line width=0.5pt] (0, 0) -- (0.5, 1.1) (0, 1.2) -- (0.5, 0.5000000000000001) (1, 1) -- (0.0, 0.6000000000000001) (0, 0) -- (1.15, 0.4) (1, 1) -- (0.6499999999999999, -0.09999999999999998) (1.3, -0.2) -- (0.4999999999999999, 0.5) (1, 1) -- (1.65, 0.4) (2, 1) -- (1.15, 0.4) (1.3, -0.2) -- (1.5, 1.0) (1, 1) -- (2.0, 1.5) (2, 1) -- (1.5, 1.5) (2, 2) -- (1.5, 1.0) (1, 1) -- (1.4, 2.1) (2, 2) -- (0.8999999999999999, 1.6) (0.8, 2.2) -- (1.5, 1.5) (2, 2) -- (1.1, 2.6) (1.4, 3) -- (1.4000000000000001, 2.1) (0.8, 2.2) -- (1.7000000000000002, 2.5) (1.4, 3) -- (0.34999999999999987, 2.35) (-0.1, 2.5) -- (1.0999999999999999, 2.6) (0.8, 2.2) -- (0.6499999999999998, 2.75) (0, 1.2) -- (0.35000000000000003, 2.35) (-0.1, 2.5) -- (0.4, 1.7000000000000002) (0.8, 2.2) -- (-0.04999999999999999, 1.85) (0, 1.2) -- (0.9, 1.6) (1, 1) -- (0.4, 1.7000000000000002) (0.8, 2.2) -- (0.5, 1.1);
\fill[cyan] (1.3,-0.2) circle (2pt) (2,2) circle (2pt) (1,1) circle (2pt) (0,1.2) circle (2pt) (1.4,3) circle (2pt) (0,0) circle (2pt) (2,1) circle (2pt) (0.8,2.2) circle (2pt) (-0.1,2.5) circle (2pt);
\fill[orange] (1.4000000000000001,2.4) circle (2pt) (1.2666666666666666,1.7333333333333334) circle (2pt) (1.4333333333333333,0.6) circle (2pt) (1.6666666666666667,1.3333333333333333) circle (2pt) (0.6,1.4666666666666668) circle (2pt) (0.23333333333333336,1.9666666666666668) circle (2pt) (0.6999999999999998,2.566666666666667) circle (2pt) (0.7666666666666666,0.26666666666666666) circle (2pt) (0.3333333333333333,0.7333333333333334) circle (2pt);
\end{scope}
\begin{scope}[shift={(6.6,0)}]
\draw (0, 0) -- (0, 1.2) (0, 0) -- (1.3, -0.2) (2, 1) -- (1.3, -0.2) (2, 1) -- (2, 2) (2, 2) -- (1.4, 3) (1.4, 3) -- (-0.1, 2.5) (0, 1.2) -- (-0.1, 2.5) (0.3333333333333333, 0.7333333333333334) -- (0.5, 1.1) (0.3333333333333333, 0.7333333333333334) -- (0.5, 0.5000000000000001) (0.3333333333333333, 0.7333333333333334) -- (0.0, 0.6000000000000001) (0.7666666666666666, 0.26666666666666666) -- (1.15, 0.4) (0.7666666666666666, 0.26666666666666666) -- (0.6499999999999999, -0.09999999999999998) (0.7666666666666666, 0.26666666666666666) -- (0.4999999999999999, 0.5) (1.4333333333333333, 0.6) -- (1.65, 0.4) (1.4333333333333333, 0.6) -- (1.15, 0.4) (1.4333333333333333, 0.6) -- (1.5, 1.0) (1.6666666666666667, 1.3333333333333333) -- (2.0, 1.5) (1.6666666666666667, 1.3333333333333333) -- (1.5, 1.5) (1.6666666666666667, 1.3333333333333333) -- (1.5, 1.0) (1.2666666666666666, 1.7333333333333334) -- (1.4, 2.1) (1.2666666666666666, 1.7333333333333334) -- (0.8999999999999999, 1.6) (1.2666666666666666, 1.7333333333333334) -- (1.5, 1.5) (1.4000000000000001, 2.4) -- (1.1, 2.6) (1.4000000000000001, 2.4) -- (1.4000000000000001, 2.1) (1.4000000000000001, 2.4) -- (1.7000000000000002, 2.5) (0.6999999999999998, 2.566666666666667) -- (0.34999999999999987, 2.35) (0.6999999999999998, 2.566666666666667) -- (1.0999999999999999, 2.6) (0.6999999999999998, 2.566666666666667) -- (0.6499999999999998, 2.75) (0.23333333333333336, 1.9666666666666668) -- (0.35000000000000003, 2.35) (0.23333333333333336, 1.9666666666666668) -- (0.4, 1.7000000000000002) (0.23333333333333336, 1.9666666666666668) -- (-0.04999999999999999, 1.85) (0.6, 1.4666666666666668) -- (0.9, 1.6) (0.6, 1.4666666666666668) -- (0.4, 1.7000000000000002) (0.6, 1.4666666666666668) -- (0.5, 1.1);
\draw[dashed,line width=0.5pt] (0, 0) -- (1, 1) (0, 1.2) -- (1, 1) (1, 1) -- (1.3, -0.2) (1, 1) -- (2, 1) (1, 1) -- (2, 2) (1, 1) -- (0.8, 2.2) (2, 2) -- (0.8, 2.2) (1.4, 3) -- (0.8, 2.2) (-0.1, 2.5) -- (0.8, 2.2) (0, 1.2) -- (0.8, 2.2) (0, 0) -- (0.3333333333333333, 0.7333333333333334) (0, 1.2) -- (0.3333333333333333, 0.7333333333333334) (1, 1) -- (0.3333333333333333, 0.7333333333333334) (0, 0) -- (0.7666666666666666, 0.26666666666666666) (1, 1) -- (0.7666666666666666, 0.26666666666666666) (1.3, -0.2) -- (0.7666666666666666, 0.26666666666666666) (1, 1) -- (1.4333333333333333, 0.6) (2, 1) -- (1.4333333333333333, 0.6) (1.3, -0.2) -- (1.4333333333333333, 0.6) (1, 1) -- (1.6666666666666667, 1.3333333333333333) (2, 1) -- (1.6666666666666667, 1.3333333333333333) (2, 2) -- (1.6666666666666667, 1.3333333333333333) (1, 1) -- (1.2666666666666666, 1.7333333333333334) (2, 2) -- (1.2666666666666666, 1.7333333333333334) (0.8, 2.2) -- (1.2666666666666666, 1.7333333333333334) (2, 2) -- (1.4000000000000001, 2.4) (1.4, 3) -- (1.4000000000000001, 2.4) (0.8, 2.2) -- (1.4000000000000001, 2.4) (1.4, 3) -- (0.6999999999999998, 2.566666666666667) (-0.1, 2.5) -- (0.6999999999999998, 2.566666666666667) (0.8, 2.2) -- (0.6999999999999998, 2.566666666666667) (0, 1.2) -- (0.23333333333333336, 1.9666666666666668) (-0.1, 2.5) -- (0.23333333333333336, 1.9666666666666668) (0.8, 2.2) -- (0.23333333333333336, 1.9666666666666668) (0, 1.2) -- (0.6, 1.4666666666666668) (1, 1) -- (0.6, 1.4666666666666668) (0.8, 2.2) -- (0.6, 1.4666666666666668);
\fill[cyan] (1.3,-0.2) circle (2pt) (2,2) circle (2pt) (1,1) circle (2pt) (0,1.2) circle (2pt) (1.4,3) circle (2pt) (0,0) circle (2pt) (2,1) circle (2pt) (0.8,2.2) circle (2pt) (-0.1,2.5) circle (2pt);
\fill[orange] (1.4000000000000001,2.4) circle (2pt) (1.2666666666666666,1.7333333333333334) circle (2pt) (1.4333333333333333,0.6) circle (2pt) (1.6666666666666667,1.3333333333333333) circle (2pt) (0.6,1.4666666666666668) circle (2pt) (0.23333333333333336,1.9666666666666668) circle (2pt) (0.6999999999999998,2.566666666666667) circle (2pt) (0.7666666666666666,0.26666666666666666) circle (2pt) (0.3333333333333333,0.7333333333333334) circle (2pt);
\end{scope}
\begin{scope}[shift={(9.9,0)}]
\draw (0, 0) -- (0, 1.2) (0, 0) -- (1.3, -0.2) (2, 1) -- (1.3, -0.2) (2, 1) -- (2, 2) (2, 2) -- (1.4, 3) (1.4, 3) -- (-0.1, 2.5) (0, 1.2) -- (-0.1, 2.5) (0.3333333333333333, 0.7333333333333334) -- (0.5, 1.1) (0.3333333333333333, 0.7333333333333334) -- (0.5, 0.5000000000000001) (0.3333333333333333, 0.7333333333333334) -- (0.0, 0.6000000000000001) (0.7666666666666666, 0.26666666666666666) -- (1.15, 0.4) (0.7666666666666666, 0.26666666666666666) -- (0.6499999999999999, -0.09999999999999998) (0.7666666666666666, 0.26666666666666666) -- (0.4999999999999999, 0.5) (1.4333333333333333, 0.6) -- (1.65, 0.4) (1.4333333333333333, 0.6) -- (1.15, 0.4) (1.4333333333333333, 0.6) -- (1.5, 1.0) (1.6666666666666667, 1.3333333333333333) -- (2.0, 1.5) (1.6666666666666667, 1.3333333333333333) -- (1.5, 1.5) (1.6666666666666667, 1.3333333333333333) -- (1.5, 1.0) (1.2666666666666666, 1.7333333333333334) -- (1.4, 2.1) (1.2666666666666666, 1.7333333333333334) -- (0.8999999999999999, 1.6) (1.2666666666666666, 1.7333333333333334) -- (1.5, 1.5) (1.4000000000000001, 2.4) -- (1.1, 2.6) (1.4000000000000001, 2.4) -- (1.4000000000000001, 2.1) (1.4000000000000001, 2.4) -- (1.7000000000000002, 2.5) (0.6999999999999998, 2.566666666666667) -- (0.34999999999999987, 2.35) (0.6999999999999998, 2.566666666666667) -- (1.0999999999999999, 2.6) (0.6999999999999998, 2.566666666666667) -- (0.6499999999999998, 2.75) (0.23333333333333336, 1.9666666666666668) -- (0.35000000000000003, 2.35) (0.23333333333333336, 1.9666666666666668) -- (0.4, 1.7000000000000002) (0.23333333333333336, 1.9666666666666668) -- (-0.04999999999999999, 1.85) (0.6, 1.4666666666666668) -- (0.9, 1.6) (0.6, 1.4666666666666668) -- (0.4, 1.7000000000000002) (0.6, 1.4666666666666668) -- (0.5, 1.1);
\fill[orange] (1.4000000000000001,2.4) circle (2pt) (1.2666666666666666,1.7333333333333334) circle (2pt) (1.4333333333333333,0.6) circle (2pt) (1.6666666666666667,1.3333333333333333) circle (2pt) (0.6,1.4666666666666668) circle (2pt) (0.23333333333333336,1.9666666666666668) circle (2pt) (0.6999999999999998,2.566666666666667) circle (2pt) (0.7666666666666666,0.26666666666666666) circle (2pt) (0.3333333333333333,0.7333333333333334) circle (2pt);
\end{scope}
\end{tikzpicture}
\caption{Creating a barycentric dual grid by refining a triangular grid. The barycentric dual grid is used in boundary element applications to define dual polynomial and
Buffa--Christiansen elements \cite{bc}.}
\label{fig:barycentric-refinement}
\end{figure}
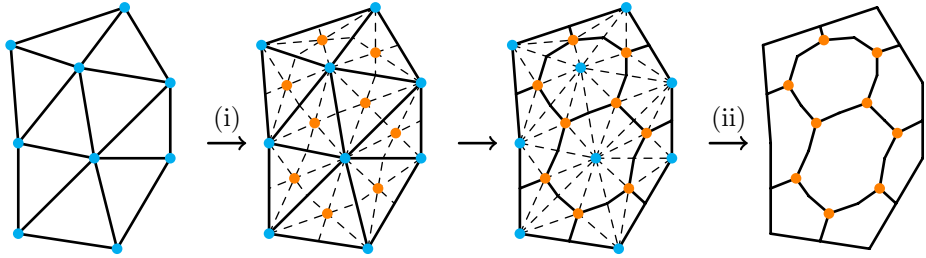

In \cite{bc}, degree 0 and 1 scalar-valued barycentric dual elements were defined alongside lowest degree $H(\sdiv)$-conforming Buffa--Christiansen elements for triangular cells.

\subsection{Element families and complexes}\label{sec:definition:complexes}
When doing finite element computations on meshes that contain a mixture of cell types, it is important to ensure that the elements
used have the correct continuity between neighbouring cells of different types.
To classify elements where this continuity is achieved, we define compatible elements as follows.

\begin{definition}[compatible elements]
Two finite elements $(\refel,\polyspace,\dualbasis)$ and $(\tilde{\refel},\tilde{\polyspace},\tilde{\dualbasis})$ are compatible if:
\begin{itemize}
\item the topological dimensions of $\refel$ and $\tilde{\refel}$ are equal;
\item for any sub-entities $\subentity\subset\refel$ and $\tilde{\subentity}\subset\tilde{\refel}$ of the same type that are not equal to the entire reference cell,
$\left.\dualbasis\right|_{\subentity}$ and
$\tilde\dualbasis\hspace{-0.25em}\left.\vphantom{\dualbasis}\right|_{\tilde{\subentity}}$ contain an equivalent set of functionals.
\end{itemize}
\end{definition}
We note that \cref{a:equivalent} can be re-expressed as the assumption that an element is compatible with itself.
Finite elements are often defined in families, with elements of the same degree within the family being compatible with each other. For example, Lagrange elements
of the same degree on different cell types are compatible.

Many PDEs of interest in mathematics and science have rich mathematical structures that are
key to their well-posedness \cite{arnold2002}. Often, this structure can be encoded in a cochain complex.
One such example is the de Rham complex, which describes the relationship between the spaces $H^1$, $H(\vcurl)$,
$H(\sdiv)$, and $L^2$, and the differential operators $\nabla$, $\nabla\times$, and $\nabla\cdot$.
In three dimensions, it is given by
$$
\mathbb{R} \hookrightarrow H^1\xrightarrow{\nabla} H(\vcurl)\xrightarrow{\nabla\times} H(\sdiv)\xrightarrow{\nabla\cdot} L^2\to0.
$$
Preserving this structure at the discrete level is often key to obtaining stable and accurate
discretisations, and there are several finite element families that replicate it. For example, the
following family forms a discrete De Rham complex in three dimensions:
$$
\parbox{1.75cm}{\centering Lagrange degree $k$}\xrightarrow{\nabla}\parbox{1.85cm}{\centering N\'ed\'elec (first kind) degree $k$}\xrightarrow{\nabla\times}\parbox{2.75cm}{\centering Raviart--Thomas degree $k$}\xrightarrow{\nabla\cdot}\parbox{2.3cm}{\centering discontinuous Lagrange degree $k-1$}
$$
Many other finite element subcomplexes exist, both for the de Rham complex as well as others.
Two naming conventions for these discrete complexes come from exterior calculus \cite{feec,periodic-table} and the work of 
Bernardo Cockburn and Guosheng Fu \cite{cockburn-fu}: in these conventions, the name of the example discrete complex above is
$\mathcal{P}_k^{-}\Lambda^r$ and $S^{\tetrahedron}_{2,k}$ (respectively) for a tetrahedron
$\mathcal{Q}_k^{-}\Lambda^r$ and $S^{\hexahedron}_{4,k}$ (respectively) for a hexahedron,
where $k$ is the degree of the element and the exterior derivative acts on $r$-forms (i.e.~$r=0$ is
Lagrange, $r=1$ is the Raviart--Thomas, and so on).

\subsection{Reference cell sub-entity numbering}\label{sec:definition:degree}
In this section, we define the reference cells that we use on DefElement, and
the numbering that we assign to their sub-entities.
We note that the choice of reference cell and numbering of its sub-entities is arbitrary,
and the definitions and results that follow could be adapted to apply to any other choice
of reference cell.
The reference cells that we use on DefElement are given in the following definition.

\begin{definition}\label{def:reference-cells}
The reference interval ($\refel^{\interval}$), triangle ($\refel^{\rtriangle}$), quadrilateral ($\refel^{\quadrilateral}$), tetrahedron ($\refel^{\tetrahedron}$),
hexahedron ($\refel^{\hexahedron}$), prism ($\refel^{\prism}$), and pyramid ($\refel^{\pyramid}$) are defined by
\begin{align*}
\refel^{\interval}&= \left\{x\,\middle|\,0\leqslant x\leqslant 1\right\},\\
\refel^{\rtriangle}&= \left\{(x,y)\,\middle|\,0\leqslant x,\,0\leqslant y,\,x+y\leqslant1\right\},\\
\refel^{\quadrilateral}&= \left\{(x,y)\,\middle|\,0\leqslant x\leqslant1,\,0\leqslant y\leqslant1\right\},\\
\refel^{\tetrahedron}&= \left\{(x,y,z)\,\middle|\,0\leqslant x,\,0\leqslant y,\,0\leqslant z,\,x+y+z\leqslant1\right\},\\
\refel^{\hexahedron}&= \left\{(x,y,z)\,\middle|\,0\leqslant x\leqslant1,\,0\leqslant y\leqslant1,\,0\leqslant z\leqslant1\right\},\\
\refel^{\prism}&= \left\{(x,y,z)\,\middle|\,0\leqslant x,\,0\leqslant y,\,x+y\leqslant1,\,0\leqslant z\leqslant1\right\},\\
\refel^{\pyramid}&= \left\{(x,y,z)\,\middle|\,0\leqslant x,\,0\leqslant y,\,0\leqslant z\leqslant1,\,x+z\leqslant1,\,y+z\leqslant1\right\}.
\end{align*}
The numbering of their sub-entities used in DefElement is shown in \cref{reference:1d-cells,reference:2d-cells,reference:3d-cells}.

For $d\in\mathbb{N}$, the reference $k$-dimensional reference simplex and cube cells
are defined by
\begin{align*}
\refel^{\simplex{d}}&=\left\{(x_0,\dots,x_{d-1})\,\middle|\,0\leqslant\min_{i=0}^{d-1}x_i\textup{ and }\sum_{i=0}^{d-1}x_i\leqslant1\right\},
\\
\refel^{\hypercube{d}}&=\left\{(x_0,\dots,x_{d-1})\,\middle|\,0\leqslant\min_{i=0}^{d-1}x_i\textup{ and }\max_{i=0}^{d-1}x_i\leqslant1\right\}.
\end{align*}
We note that for $d=1$ to $d=3$, the simplex reference cells correspond to
$\refel^{\interval}$,
$\refel^{\rtriangle}$ and
$\refel^{\tetrahedron}$, and the cube cells correspond to
$\refel^{\interval}$,
$\refel^{\quadrilateral}$ and
$\refel^{\hexahedron}$.
\end{definition}

In order to define a consistent convention to use for numbering the sub-entities of reference
cells on DefElement, 
we use the following definition of $<$ for tuples of values, which corresponds to
the behaviour of the \texttt{<} operator in Python.

\begin{definition}
For two tuples $a$ and $b$, $<$ is defined by:
\begin{itemize}
\item For two length 1 tuples $a=(a_0)$ and $b=(b_0)$,
$a<b$ if and only if
$a_0<b_0$
\item For two length $n$ tuples $a=(a_0,\dots,a_{n-1})$ and $b=(b_0,\dots,b_{n-1})$,
$a<b$ if and only if one of the following conditions holds:
\begin{itemize}
\item $a_0<b_0$,
\item $a_0=b_0$ and $(a_1,\dots,a_{n-1}) < (b_1,\dots,b_{n-1})$.
\end{itemize}
\item For two tuples $a$ and $b$ of length $n_a$ and $n_b$,
$a<b$ if and only if one of the following conditions holds:
\begin{itemize}
\item $(a_0,\dots,a_{\min(n_a, n_b)})<(b_0,\dots,b_{\min(n_a, n_b)})$,
\item $(a_0,\dots,a_{\min(n_a, n_b)})=(b_0,\dots,b_{\min(n_a, n_b)})$ and $n_a < n_b$.
\end{itemize}
\end{itemize}
\end{definition}

We now present the sub-entity numbering conventions used by DefElement. The aim of these is to
ensure consistency in the choices of numbering. These conventions lead the numbering
shown in \cref{reference:1d-cells,reference:2d-cells,reference:3d-cells}.

\begin{convention}[Numbering of vertices]\label{c:vertices}
Let $a=(a_0,\dots,a_{d-1}), b=(b_0,\dots,b_{d-1})\in\mathbb{R}^{d}$ be two vertices of a reference
cell $\refel$. The index of vertex $a$ is less than the index of vertex $b$ if and only if
$(a_{d-1},\dots,a_0)<(b_{d-1},\dots,b_0)$.
\end{convention}

\begin{convention}[Numbering of sub-entities]\label{c:sub-entities}
Let $a, b\subseteq\refel$ be two sub-entities of the same topological dimension.
Let $v_a$ and $v_b$ be tuples containing the indices of vertices connected to $a$ and $b$
(respectively), with vertices with lower indices appearing earlier in the tuples.
The index of sub-entity $a$ is less than the index of sub-entity $b$ if and only if
$v_a < v_b$.
\end{convention}

\begin{figure}[!ht]
\centering
\begin{tikzpicture}[x=1cm,y=1cm,line cap=round,line join=round]
\footnotesize
\draw[-stealth,black,line width=0.8pt,line cap=round] (25.0,25.0) -- (25.375,25.0);
\node[black,anchor=west] at (25.405,25.0) {$x$};\draw[black,line width=0.8pt,line cap=round] (26.125,25.0) -- (27.25,25.0);
\draw[orange,line width=0.8pt,fill=white] (26.125,25.0) circle (4.0pt);
\node[black,anchor=center] at (26.125,25.0) {0};\draw[orange,line width=0.8pt,fill=white] (27.25,25.0) circle (4.0pt);
\node[black,anchor=center] at (27.25,25.0) {1};\draw[black,line width=0.8pt,line cap=round] (27.625,25.0) -- (28.75,25.0);
\draw[cyan,line width=0.8pt,fill=white] (28.1875,25.0) circle (4.0pt);
\node[black,anchor=center] at (28.1875,25.0) {0};\end{tikzpicture}
\caption{The numbering of the sub-entities that DefElement uses for the reference interval.}
\label{reference:1d-cells}
\end{figure}
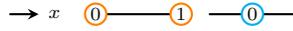

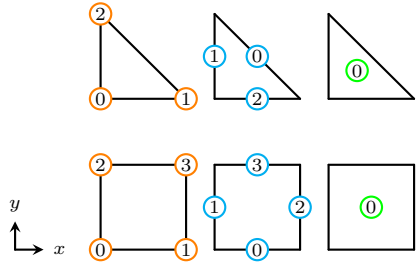
\begin{figure}[!ht]
\centering
\begin{tikzpicture}[x=1cm,y=1cm,line cap=round,line join=round]
\footnotesize
\draw[black,line width=0.8pt,line cap=round] (27.25,25.0) -- (26.125,26.125);
\draw[black,line width=0.8pt,line cap=round] (26.125,25.0) -- (26.125,26.125);
\draw[black,line width=0.8pt,line cap=round] (26.125,25.0) -- (27.25,25.0);
\draw[orange,line width=0.8pt,fill=white] (26.125,25.0) circle (4.0pt);
\node[black,anchor=center] at (26.125,25.0) {0};\draw[orange,line width=0.8pt,fill=white] (27.25,25.0) circle (4.0pt);
\node[black,anchor=center] at (27.25,25.0) {1};\draw[orange,line width=0.8pt,fill=white] (26.125,26.125) circle (4.0pt);
\node[black,anchor=center] at (26.125,26.125) {2};\draw[black,line width=0.8pt,line cap=round] (28.75,25.0) -- (27.625,26.125);
\draw[cyan,line width=0.8pt,fill=white] (28.1875,25.5625) circle (4.0pt);
\node[black,anchor=center] at (28.1875,25.5625) {2};\draw[black,line width=0.8pt,line cap=round] (27.625,25.0) -- (27.625,26.125);
\draw[cyan,line width=0.8pt,fill=white] (27.625,25.5625) circle (4.0pt);
\node[black,anchor=center] at (27.625,25.5625) {1};\draw[black,line width=0.8pt,line cap=round] (27.625,25.0) -- (28.75,25.0);
\draw[cyan,line width=0.8pt,fill=white] (28.1875,25.0) circle (4.0pt);
\node[black,anchor=center] at (28.1875,25.0) {0};\draw[green,line width=0.8pt,fill=white] (29.5,25.375) circle (4.0pt);
\node[black,anchor=center] at (29.5,25.375) {0};\draw[black,line width=0.8pt,line cap=round] (30.25,25.0) -- (29.125,26.125);
\draw[black,line width=0.8pt,line cap=round] (29.125,25.0) -- (29.125,26.125);
\draw[black,line width=0.8pt,line cap=round] (29.125,25.0) -- (30.25,25.0);
\begin{scope}[shift={(0,-2)}]
\draw[-stealth,black,line width=0.8pt,line cap=round] (25.0,25.0) -- (25.375,25.0);
\draw[-stealth,black,line width=0.8pt,line cap=round] (25.0,25.0) -- (25.0,25.375);
\node[black,anchor=west] at (25.405,25.0) {$x$};\node[black,anchor=south] at (25.0,25.405) {$y$};\draw[black,line width=0.8pt,line cap=round] (26.125,25.0) -- (27.25,25.0);
\draw[black,line width=0.8pt,line cap=round] (26.125,25.0) -- (26.125,26.125);
\draw[black,line width=0.8pt,line cap=round] (27.25,25.0) -- (27.25,26.125);
\draw[black,line width=0.8pt,line cap=round] (26.125,26.125) -- (27.25,26.125);
\draw[orange,line width=0.8pt,fill=white] (26.125,25.0) circle (4.0pt);
\node[black,anchor=center] at (26.125,25.0) {0};\draw[orange,line width=0.8pt,fill=white] (27.25,25.0) circle (4.0pt);
\node[black,anchor=center] at (27.25,25.0) {1};\draw[orange,line width=0.8pt,fill=white] (26.125,26.125) circle (4.0pt);
\node[black,anchor=center] at (26.125,26.125) {2};\draw[orange,line width=0.8pt,fill=white] (27.25,26.125) circle (4.0pt);
\node[black,anchor=center] at (27.25,26.125) {3};\draw[black,line width=0.8pt,line cap=round] (27.625,25.0) -- (28.75,25.0);
\draw[cyan,line width=0.8pt,fill=white] (28.1875,25.0) circle (4.0pt);
\node[black,anchor=center] at (28.1875,25.0) {0};\draw[black,line width=0.8pt,line cap=round] (27.625,25.0) -- (27.625,26.125);
\draw[cyan,line width=0.8pt,fill=white] (27.625,25.5625) circle (4.0pt);
\node[black,anchor=center] at (27.625,25.5625) {1};\draw[black,line width=0.8pt,line cap=round] (28.75,25.0) -- (28.75,26.125);
\draw[cyan,line width=0.8pt,fill=white] (28.75,25.5625) circle (4.0pt);
\node[black,anchor=center] at (28.75,25.5625) {2};\draw[black,line width=0.8pt,line cap=round] (27.625,26.125) -- (28.75,26.125);
\draw[cyan,line width=0.8pt,fill=white] (28.1875,26.125) circle (4.0pt);
\node[black,anchor=center] at (28.1875,26.125) {3};\draw[green,line width=0.8pt,fill=white] (29.6875,25.5625) circle (4.0pt);
\node[black,anchor=center] at (29.6875,25.5625) {0};\draw[black,line width=0.8pt,line cap=round] (29.125,25.0) -- (30.25,25.0);
\draw[black,line width=0.8pt,line cap=round] (29.125,25.0) -- (29.125,26.125);
\draw[black,line width=0.8pt,line cap=round] (30.25,25.0) -- (30.25,26.125);
\draw[black,line width=0.8pt,line cap=round] (29.125,26.125) -- (30.25,26.125);
\end{scope}
\end{tikzpicture}
\caption{The numbering of the sub-entities that DefElement uses for two-dimensional reference cells.}
\label{reference:2d-cells}
\end{figure}

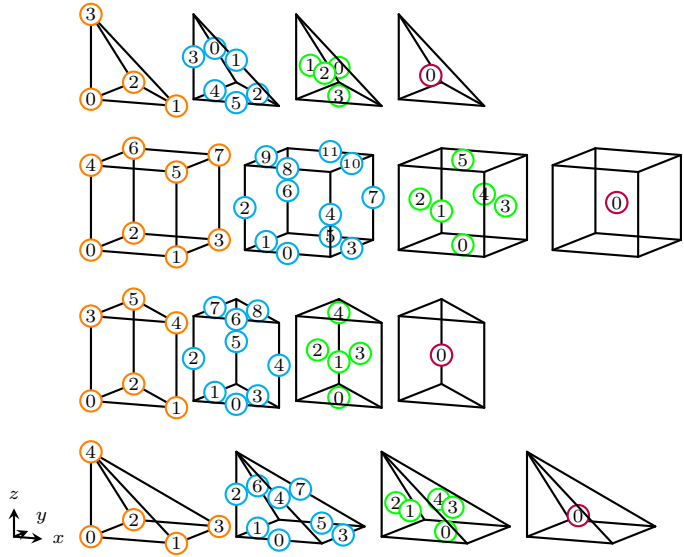
\begin{figure}[!ht]
\centering
\begin{tikzpicture}[x=1cm,y=1cm,line cap=round,line join=round]
\footnotesize
\draw[black,line width=0.8pt,line cap=round] (26.575,25.315) -- (26.0125,26.215);
\draw[black,line width=0.8pt,line cap=round] (27.1375,25.0) -- (26.575,25.315);
\draw[black,line width=0.8pt,line cap=round] (26.0125,25.09) -- (26.575,25.315);
\draw[orange,line width=0.8pt,fill=white] (26.575,25.315) circle (4.0pt);
\node[black,anchor=center] at (26.575,25.315) {2};\draw[black,line width=0.8pt,line cap=round] (27.1375,25.0) -- (26.0125,26.215);
\draw[black,line width=0.8pt,line cap=round] (26.0125,25.09) -- (26.0125,26.215);
\draw[black,line width=0.8pt,line cap=round] (26.0125,25.09) -- (27.1375,25.0);
\draw[orange,line width=0.8pt,fill=white] (26.0125,25.09) circle (4.0pt);
\node[black,anchor=center] at (26.0125,25.09) {0};\draw[orange,line width=0.8pt,fill=white] (27.1375,25.0) circle (4.0pt);
\node[black,anchor=center] at (27.1375,25.0) {1};\draw[orange,line width=0.8pt,fill=white] (26.0125,26.215) circle (4.0pt);
\node[black,anchor=center] at (26.0125,26.215) {3};\draw[black,line width=0.8pt,line cap=round] (27.925,25.315) -- (27.3625,26.215);
\draw[cyan,line width=0.8pt,fill=white] (27.64375,25.765) circle (4.0pt);
\node[black,anchor=center] at (27.64375,25.765) {5};\draw[black,line width=0.8pt,line cap=round] (28.4875,25.0) -- (27.925,25.315);
\draw[cyan,line width=0.8pt,fill=white] (28.20625,25.1575) circle (4.0pt);
\node[black,anchor=center] at (28.20625,25.1575) {3};\draw[black,line width=0.8pt,line cap=round] (27.3625,25.09) -- (27.925,25.315);
\draw[cyan,line width=0.8pt,fill=white] (27.64375,25.2025) circle (4.0pt);
\node[black,anchor=center] at (27.64375,25.2025) {1};\draw[black,line width=0.8pt,line cap=round] (28.4875,25.0) -- (27.3625,26.215);
\draw[cyan,line width=0.8pt,fill=white] (27.925,25.6075) circle (4.0pt);
\node[black,anchor=center] at (27.925,25.6075) {4};\draw[black,line width=0.8pt,line cap=round] (27.3625,25.09) -- (27.3625,26.215);
\draw[cyan,line width=0.8pt,fill=white] (27.3625,25.6525) circle (4.0pt);
\node[black,anchor=center] at (27.3625,25.6525) {2};\draw[black,line width=0.8pt,line cap=round] (27.3625,25.09) -- (28.4875,25.0);
\draw[cyan,line width=0.8pt,fill=white] (27.925,25.045) circle (4.0pt);
\node[black,anchor=center] at (27.925,25.045) {0};\draw[green,line width=0.8pt,fill=white] (29.275,25.51) circle (4.0pt);
\node[black,anchor=center] at (29.275,25.51) {3};\draw[green,line width=0.8pt,fill=white] (28.9,25.54) circle (4.0pt);
\node[black,anchor=center] at (28.9,25.54) {2};\draw[green,line width=0.8pt,fill=white] (29.275,25.135) circle (4.0pt);
\node[black,anchor=center] at (29.275,25.135) {0};\draw[black,line width=0.8pt,line cap=round] (29.275,25.315) -- (28.7125,26.215);
\draw[black,line width=0.8pt,line cap=round] (29.8375,25.0) -- (29.275,25.315);
\draw[black,line width=0.8pt,line cap=round] (28.7125,25.09) -- (29.275,25.315);
\draw[green,line width=0.8pt,fill=white] (29.0875,25.435) circle (4.0pt);
\node[black,anchor=center] at (29.0875,25.435) {1};\draw[black,line width=0.8pt,line cap=round] (29.8375,25.0) -- (28.7125,26.215);
\draw[black,line width=0.8pt,line cap=round] (28.7125,25.09) -- (28.7125,26.215);
\draw[black,line width=0.8pt,line cap=round] (28.7125,25.09) -- (29.8375,25.0);
\draw[black,line width=0.8pt,line cap=round] (30.625,25.315) -- (30.0625,26.215);
\draw[black,line width=0.8pt,line cap=round] (31.1875,25.0) -- (30.625,25.315);
\draw[black,line width=0.8pt,line cap=round] (30.0625,25.09) -- (30.625,25.315);
\draw[purple,line width=0.8pt,fill=white] (30.484375,25.405) circle (4.0pt);
\node[black,anchor=center] at (30.484375,25.405) {0};\draw[black,line width=0.8pt,line cap=round] (31.1875,25.0) -- (30.0625,26.215);
\draw[black,line width=0.8pt,line cap=round] (30.0625,25.09) -- (30.0625,26.215);
\draw[black,line width=0.8pt,line cap=round] (30.0625,25.09) -- (31.1875,25.0);
\begin{scope}[shift={(0,-2)}]
\draw[black,line width=0.8pt,line cap=round] (26.575,25.315) -- (27.7,25.225);
\draw[black,line width=0.8pt,line cap=round] (26.575,25.315) -- (26.575,26.44);
\draw[black,line width=0.8pt,line cap=round] (26.0125,25.09) -- (26.575,25.315);
\draw[orange,line width=0.8pt,fill=white] (26.575,25.315) circle (4.0pt);
\node[black,anchor=center] at (26.575,25.315) {2};\draw[black,line width=0.8pt,line cap=round] (27.1375,25.0) -- (27.7,25.225);
\draw[black,line width=0.8pt,line cap=round] (27.7,25.225) -- (27.7,26.35);
\draw[orange,line width=0.8pt,fill=white] (27.7,25.225) circle (4.0pt);
\node[black,anchor=center] at (27.7,25.225) {3};\draw[black,line width=0.8pt,line cap=round] (26.0125,25.09) -- (27.1375,25.0);
\draw[black,line width=0.8pt,line cap=round] (26.0125,25.09) -- (26.0125,26.215);
\draw[black,line width=0.8pt,line cap=round] (27.1375,25.0) -- (27.1375,26.125);
\draw[orange,line width=0.8pt,fill=white] (26.0125,25.09) circle (4.0pt);
\node[black,anchor=center] at (26.0125,25.09) {0};\draw[orange,line width=0.8pt,fill=white] (27.1375,25.0) circle (4.0pt);
\node[black,anchor=center] at (27.1375,25.0) {1};\draw[black,line width=0.8pt,line cap=round] (26.0125,26.215) -- (27.1375,26.125);
\draw[black,line width=0.8pt,line cap=round] (26.0125,26.215) -- (26.575,26.44);
\draw[black,line width=0.8pt,line cap=round] (27.1375,26.125) -- (27.7,26.35);
\draw[black,line width=0.8pt,line cap=round] (26.575,26.44) -- (27.7,26.35);
\draw[orange,line width=0.8pt,fill=white] (26.0125,26.215) circle (4.0pt);
\node[black,anchor=center] at (26.0125,26.215) {4};\draw[orange,line width=0.8pt,fill=white] (27.1375,26.125) circle (4.0pt);
\node[black,anchor=center] at (27.1375,26.125) {5};\draw[orange,line width=0.8pt,fill=white] (26.575,26.44) circle (4.0pt);
\node[black,anchor=center] at (26.575,26.44) {6};\draw[orange,line width=0.8pt,fill=white] (27.7,26.35) circle (4.0pt);
\node[black,anchor=center] at (27.7,26.35) {7};\draw[black,line width=0.8pt,line cap=round] (28.6,25.315) -- (29.725,25.225);
\draw[cyan,line width=0.8pt,fill=white] (29.1625,25.27) circle (4.0pt);
\node[black,anchor=center] at (29.1625,25.27) {5};\draw[black,line width=0.8pt,line cap=round] (28.6,25.315) -- (28.6,26.44);
\draw[cyan,line width=0.8pt,fill=white] (28.6,25.8775) circle (4.0pt);
\node[black,anchor=center] at (28.6,25.8775) {6};\draw[black,line width=0.8pt,line cap=round] (28.0375,25.09) -- (28.6,25.315);
\draw[cyan,line width=0.8pt,fill=white] (28.31875,25.2025) circle (4.0pt);
\node[black,anchor=center] at (28.31875,25.2025) {1};\draw[black,line width=0.8pt,line cap=round] (29.1625,25.0) -- (29.725,25.225);
\draw[cyan,line width=0.8pt,fill=white] (29.44375,25.1125) circle (4.0pt);
\node[black,anchor=center] at (29.44375,25.1125) {3};\draw[black,line width=0.8pt,line cap=round] (29.725,25.225) -- (29.725,26.35);
\draw[cyan,line width=0.8pt,fill=white] (29.725,25.7875) circle (4.0pt);
\node[black,anchor=center] at (29.725,25.7875) {7};\draw[black,line width=0.8pt,line cap=round] (28.0375,25.09) -- (29.1625,25.0);
\draw[cyan,line width=0.8pt,fill=white] (28.6,25.045) circle (4.0pt);
\node[black,anchor=center] at (28.6,25.045) {0};\draw[black,line width=0.8pt,line cap=round] (28.0375,25.09) -- (28.0375,26.215);
\draw[cyan,line width=0.8pt,fill=white] (28.0375,25.6525) circle (4.0pt);
\node[black,anchor=center] at (28.0375,25.6525) {2};\draw[black,line width=0.8pt,line cap=round] (29.1625,25.0) -- (29.1625,26.125);
\draw[cyan,line width=0.8pt,fill=white] (29.1625,25.5625) circle (4.0pt);
\node[black,anchor=center] at (29.1625,25.5625) {4};
\draw[black,line width=0.8pt,line cap=round] (28.0375,26.215) -- (28.6,26.44);
\draw[cyan,line width=0.8pt,fill=white] (28.31875,26.3275) circle (4.0pt);
\node[black,anchor=center] at (28.31875,26.3275) {9};
\draw[black,line width=0.8pt,line cap=round] (28.0375,26.215) -- (29.1625,26.125);
\draw[cyan,line width=0.8pt,fill=white] (28.6,26.17) circle (4.0pt);
\node[black,anchor=center] at (28.6,26.17) {8};
\draw[black,line width=0.8pt,line cap=round] (28.6,26.44) -- (29.725,26.35);
\draw[cyan,line width=0.8pt,fill=white] (29.1625,26.395) circle (4.0pt);
\node[black,anchor=center] at (29.1625,26.395) {\tiny 11};
\draw[black,line width=0.8pt,line cap=round] (29.1625,26.125) -- (29.725,26.35);
\draw[cyan,line width=0.8pt,fill=white] (29.44375,26.2375) circle (4.0pt);
\node[black,anchor=center] at (29.44375,26.2375) {\tiny 10};
\draw[green,line width=0.8pt,fill=white] (31.1875,25.8325) circle (4.0pt);
\node[black,anchor=center] at (31.1875,25.8325) {4};
\draw[black,line width=0.8pt,line cap=round] (30.625,25.315) -- (31.75,25.225);
\draw[black,line width=0.8pt,line cap=round] (30.625,25.315) -- (30.625,26.44);
\draw[green,line width=0.8pt,fill=white] (30.90625,25.1575) circle (4.0pt);
\draw[black,line width=0.8pt,line cap=round] (31.1875,25.0) -- (31.75,25.225);
\draw[black,line width=0.8pt,line cap=round] (30.0625,25.09) -- (31.1875,25.0);
\draw[black,line width=0.8pt,line cap=round] (30.0625,25.09) -- (30.625,25.315);
\draw[black,line width=0.8pt,line cap=round] (30.0625,26.215) -- (30.625,26.44);
\node[black,anchor=center] at (30.90625,25.1575) {0};
\draw[green,line width=0.8pt,fill=white] (30.34375,25.765) circle (4.0pt);
\node[black,anchor=center] at (30.34375,25.765) {2};
\draw[green,line width=0.8pt,fill=white] (31.46875,25.675) circle (4.0pt);
\node[black,anchor=center] at (31.46875,25.675) {3};
\draw[black,line width=0.8pt,line cap=round] (31.75,25.225) -- (31.75,26.35);
\draw[black,line width=0.8pt,line cap=round] (30.0625,25.09) -- (30.0625,26.215);
\draw[black,line width=0.8pt,line cap=round] (31.1875,25.0) -- (31.1875,26.125);
\draw[black,line width=0.8pt,line cap=round] (31.1875,26.125) -- (31.75,26.35);
\draw[black,line width=0.8pt,line cap=round] (30.625,26.44) -- (31.75,26.35);
\draw[black,line width=0.8pt,line cap=round] (32.65,25.315) -- (33.775,25.225);
\draw[green,line width=0.8pt,fill=white] (30.625,25.6075) circle (4.0pt);
\node[black,anchor=center] at (30.625,25.6075) {1};
\draw[green,line width=0.8pt,fill=white] (30.90625,26.2825) circle (4.0pt);
\node[black,anchor=center] at (30.90625,26.2825) {5};
\draw[black,line width=0.8pt,line cap=round] (30.0625,26.215) -- (31.1875,26.125);
\draw[black,line width=0.8pt,line cap=round] (32.65,25.315) -- (32.65,26.44);
\draw[black,line width=0.8pt,line cap=round] (32.0875,25.09) -- (32.65,25.315);
\draw[purple,line width=0.8pt,fill=white] (32.93125,25.72) circle (4.0pt);
\node[black,anchor=center] at (32.93125,25.72) {0};\draw[black,line width=0.8pt,line cap=round] (33.2125,25.0) -- (33.775,25.225);
\draw[black,line width=0.8pt,line cap=round] (33.775,25.225) -- (33.775,26.35);
\draw[black,line width=0.8pt,line cap=round] (32.0875,25.09) -- (33.2125,25.0);
\draw[black,line width=0.8pt,line cap=round] (32.0875,25.09) -- (32.0875,26.215);
\draw[black,line width=0.8pt,line cap=round] (33.2125,25.0) -- (33.2125,26.125);
\draw[black,line width=0.8pt,line cap=round] (32.0875,26.215) -- (33.2125,26.125);
\draw[black,line width=0.8pt,line cap=round] (32.0875,26.215) -- (32.65,26.44);
\draw[black,line width=0.8pt,line cap=round] (33.2125,26.125) -- (33.775,26.35);
\draw[black,line width=0.8pt,line cap=round] (32.65,26.44) -- (33.775,26.35);
\end{scope}
\begin{scope}[shift={(0,-4)}]
\draw[black,line width=0.8pt,line cap=round] (27.1375,25.0) -- (26.575,25.315);
\draw[black,line width=0.8pt,line cap=round] (26.0125,25.09) -- (26.575,25.315);
\draw[black,line width=0.8pt,line cap=round] (26.575,25.315) -- (26.575,26.44);
\draw[orange,line width=0.8pt,fill=white] (26.575,25.315) circle (4.0pt);
\node[black,anchor=center] at (26.575,25.315) {2};\draw[black,line width=0.8pt,line cap=round] (26.0125,25.09) -- (27.1375,25.0);
\draw[black,line width=0.8pt,line cap=round] (26.0125,25.09) -- (26.0125,26.215);
\draw[black,line width=0.8pt,line cap=round] (27.1375,25.0) -- (27.1375,26.125);
\draw[orange,line width=0.8pt,fill=white] (26.0125,25.09) circle (4.0pt);
\node[black,anchor=center] at (26.0125,25.09) {0};\draw[orange,line width=0.8pt,fill=white] (27.1375,25.0) circle (4.0pt);
\node[black,anchor=center] at (27.1375,25.0) {1};\draw[black,line width=0.8pt,line cap=round] (26.0125,26.215) -- (27.1375,26.125);
\draw[black,line width=0.8pt,line cap=round] (26.0125,26.215) -- (26.575,26.44);
\draw[black,line width=0.8pt,line cap=round] (27.1375,26.125) -- (26.575,26.44);
\draw[orange,line width=0.8pt,fill=white] (26.0125,26.215) circle (4.0pt);
\node[black,anchor=center] at (26.0125,26.215) {3};\draw[orange,line width=0.8pt,fill=white] (27.1375,26.125) circle (4.0pt);
\node[black,anchor=center] at (27.1375,26.125) {4};\draw[orange,line width=0.8pt,fill=white] (26.575,26.44) circle (4.0pt);
\node[black,anchor=center] at (26.575,26.44) {5};\draw[black,line width=0.8pt,line cap=round] (28.4875,25.0) -- (27.925,25.315);
\draw[cyan,line width=0.8pt,fill=white] (28.20625,25.1575) circle (4.0pt);
\node[black,anchor=center] at (28.20625,25.1575) {3};\draw[black,line width=0.8pt,line cap=round] (27.3625,25.09) -- (27.925,25.315);
\draw[cyan,line width=0.8pt,fill=white] (27.64375,25.2025) circle (4.0pt);
\node[black,anchor=center] at (27.64375,25.2025) {1};\draw[black,line width=0.8pt,line cap=round] (27.925,25.315) -- (27.925,26.44);
\draw[cyan,line width=0.8pt,fill=white] (27.925,25.8775) circle (4.0pt);
\node[black,anchor=center] at (27.925,25.8775) {5};\draw[black,line width=0.8pt,line cap=round] (27.3625,25.09) -- (28.4875,25.0);
\draw[cyan,line width=0.8pt,fill=white] (27.925,25.045) circle (4.0pt);
\node[black,anchor=center] at (27.925,25.045) {0};\draw[black,line width=0.8pt,line cap=round] (27.3625,25.09) -- (27.3625,26.215);
\draw[cyan,line width=0.8pt,fill=white] (27.3625,25.6525) circle (4.0pt);
\node[black,anchor=center] at (27.3625,25.6525) {2};\draw[black,line width=0.8pt,line cap=round] (28.4875,25.0) -- (28.4875,26.125);
\draw[cyan,line width=0.8pt,fill=white] (28.4875,25.5625) circle (4.0pt);
\node[black,anchor=center] at (28.4875,25.5625) {4};
\draw[black,line width=0.8pt,line cap=round] (27.3625,26.215) -- (27.925,26.44);
\draw[cyan,line width=0.8pt,fill=white] (27.64375,26.3275) circle (4.0pt);
\node[black,anchor=center] at (27.64375,26.3275) {7};\draw[black,line width=0.8pt,line cap=round] (28.4875,26.125) -- (27.925,26.44);
\draw[cyan,line width=0.8pt,fill=white] (28.20625,26.2825) circle (4.0pt);
\node[black,anchor=center] at (28.20625,26.2825) {8};
\draw[black,line width=0.8pt,line cap=round] (27.3625,26.215) -- (28.4875,26.125);
\draw[cyan,line width=0.8pt,fill=white] (27.925,26.17) circle (4.0pt);
\node[black,anchor=center] at (27.925,26.17) {6};
\draw[green,line width=0.8pt,fill=white] (29.275,25.135) circle (4.0pt);
\node[black,anchor=center] at (29.275,25.135) {0};\draw[green,line width=0.8pt,fill=white] (29.55625,25.72) circle (4.0pt);
\node[black,anchor=center] at (29.55625,25.72) {3};\draw[black,line width=0.8pt,line cap=round] (29.8375,25.0) -- (29.275,25.315);
\draw[green,line width=0.8pt,fill=white] (28.99375,25.765) circle (4.0pt);
\node[black,anchor=center] at (28.99375,25.765) {2};\draw[black,line width=0.8pt,line cap=round] (28.7125,25.09) -- (29.275,25.315);
\draw[black,line width=0.8pt,line cap=round] (29.275,25.315) -- (29.275,26.44);
\draw[green,line width=0.8pt,fill=white] (29.275,25.6075) circle (4.0pt);
\node[black,anchor=center] at (29.275,25.6075) {1};\draw[black,line width=0.8pt,line cap=round] (28.7125,25.09) -- (29.8375,25.0);
\draw[black,line width=0.8pt,line cap=round] (28.7125,25.09) -- (28.7125,26.215);
\draw[black,line width=0.8pt,line cap=round] (29.8375,25.0) -- (29.8375,26.125);
\draw[green,line width=0.8pt,fill=white] (29.275,26.26) circle (4.0pt);
\node[black,anchor=center] at (29.275,26.26) {4};\draw[black,line width=0.8pt,line cap=round] (28.7125,26.215) -- (29.8375,26.125);
\draw[black,line width=0.8pt,line cap=round] (28.7125,26.215) -- (29.275,26.44);
\draw[black,line width=0.8pt,line cap=round] (29.8375,26.125) -- (29.275,26.44);
\draw[black,line width=0.8pt,line cap=round] (31.1875,25.0) -- (30.625,25.315);
\draw[black,line width=0.8pt,line cap=round] (30.0625,25.09) -- (30.625,25.315);
\draw[black,line width=0.8pt,line cap=round] (30.625,25.315) -- (30.625,26.44);
\draw[purple,line width=0.8pt,fill=white] (30.625,25.6975) circle (4.0pt);
\node[black,anchor=center] at (30.625,25.6975) {0};\draw[black,line width=0.8pt,line cap=round] (30.0625,25.09) -- (31.1875,25.0);
\draw[black,line width=0.8pt,line cap=round] (30.0625,25.09) -- (30.0625,26.215);
\draw[black,line width=0.8pt,line cap=round] (31.1875,25.0) -- (31.1875,26.125);
\draw[black,line width=0.8pt,line cap=round] (30.0625,26.215) -- (31.1875,26.125);
\draw[black,line width=0.8pt,line cap=round] (30.0625,26.215) -- (30.625,26.44);
\draw[black,line width=0.8pt,line cap=round] (31.1875,26.125) -- (30.625,26.44);
\end{scope}
\begin{scope}[shift={(0,-5.8)}]
\draw[-stealth,black,line width=0.8pt,line cap=round] (25.0,25.09) -- (25.375,25.06);
\draw[-stealth,black,line width=0.8pt,line cap=round] (25.0,25.09) -- (25.1875,25.165);
\draw[-stealth,black,line width=0.8pt,line cap=round] (25.0,25.09) -- (25.0,25.465);
\node[black,anchor=west] at (25.405,25.0576) {$x$};\node[black,anchor=south west] at (25.1875,25.165) {$y$};\node[black,anchor=south] at (25.0,25.495) {$z$};\draw[black,line width=0.8pt,line cap=round] (26.575,25.315) -- (27.7,25.225);
\draw[black,line width=0.8pt,line cap=round] (26.0125,25.09) -- (26.575,25.315);
\draw[black,line width=0.8pt,line cap=round] (26.575,25.315) -- (26.0125,26.215);
\draw[orange,line width=0.8pt,fill=white] (26.575,25.315) circle (4.0pt);
\node[black,anchor=center] at (26.575,25.315) {2};\draw[black,line width=0.8pt,line cap=round] (27.1375,25.0) -- (27.7,25.225);
\draw[black,line width=0.8pt,line cap=round] (27.7,25.225) -- (26.0125,26.215);
\draw[orange,line width=0.8pt,fill=white] (27.7,25.225) circle (4.0pt);
\node[black,anchor=center] at (27.7,25.225) {3};\draw[black,line width=0.8pt,line cap=round] (26.0125,25.09) -- (27.1375,25.0);
\draw[black,line width=0.8pt,line cap=round] (26.0125,25.09) -- (26.0125,26.215);
\draw[black,line width=0.8pt,line cap=round] (27.1375,25.0) -- (26.0125,26.215);
\draw[orange,line width=0.8pt,fill=white] (26.0125,25.09) circle (4.0pt);
\node[black,anchor=center] at (26.0125,25.09) {0};\draw[orange,line width=0.8pt,fill=white] (27.1375,25.0) circle (4.0pt);
\node[black,anchor=center] at (27.1375,25.0) {1};\draw[orange,line width=0.8pt,fill=white] (26.0125,26.215) circle (4.0pt);
\node[black,anchor=center] at (26.0125,26.215) {4};\draw[black,line width=0.8pt,line cap=round] (28.4875,25.315) -- (29.6125,25.225);
\draw[cyan,line width=0.8pt,fill=white] (29.05,25.27) circle (4.0pt);
\node[black,anchor=center] at (29.05,25.27) {5};\draw[black,line width=0.8pt,line cap=round] (27.925,25.09) -- (28.4875,25.315);
\draw[cyan,line width=0.8pt,fill=white] (28.20625,25.2025) circle (4.0pt);
\node[black,anchor=center] at (28.20625,25.2025) {1};\draw[black,line width=0.8pt,line cap=round] (28.4875,25.315) -- (27.925,26.215);
\draw[cyan,line width=0.8pt,fill=white] (28.20625,25.765) circle (4.0pt);
\node[black,anchor=center] at (28.20625,25.765) {6};\draw[black,line width=0.8pt,line cap=round] (29.05,25.0) -- (29.6125,25.225);
\draw[cyan,line width=0.8pt,fill=white] (29.33125,25.1125) circle (4.0pt);
\node[black,anchor=center] at (29.33125,25.1125) {3};\draw[black,line width=0.8pt,line cap=round] (29.6125,25.225) -- (27.925,26.215);
\draw[cyan,line width=0.8pt,fill=white] (28.76875,25.72) circle (4.0pt);
\node[black,anchor=center] at (28.76875,25.72) {7};\draw[black,line width=0.8pt,line cap=round] (27.925,25.09) -- (29.05,25.0);
\draw[cyan,line width=0.8pt,fill=white] (28.4875,25.045) circle (4.0pt);
\node[black,anchor=center] at (28.4875,25.045) {0};\draw[black,line width=0.8pt,line cap=round] (27.925,25.09) -- (27.925,26.215);
\draw[cyan,line width=0.8pt,fill=white] (27.925,25.6525) circle (4.0pt);
\node[black,anchor=center] at (27.925,25.6525) {2};\draw[black,line width=0.8pt,line cap=round] (29.05,25.0) -- (27.925,26.215);
\draw[cyan,line width=0.8pt,fill=white] (28.4875,25.6075) circle (4.0pt);
\node[black,anchor=center] at (28.4875,25.6075) {4};
\draw[black,line width=0.8pt,line cap=round] (30.4,25.315) -- (31.525,25.225);
\draw[green,line width=0.8pt,fill=white] (30.68125,25.1575) circle (4.0pt);
\node[black,anchor=center] at (30.68125,25.1575) {0};
\draw[green,line width=0.8pt,fill=white] (30.5875,25.585) circle (4.0pt);
\node[black,anchor=center] at (30.5875,25.585) {4};
\draw[green,line width=0.8pt,fill=white] (30.025,25.54) circle (4.0pt);
\node[black,anchor=center] at (30.025,25.54) {2};\draw[black,line width=0.8pt,line cap=round] (29.8375,25.09) -- (30.4,25.315);
\draw[black,line width=0.8pt,line cap=round] (30.4,25.315) -- (29.8375,26.215);
\draw[green,line width=0.8pt,fill=white] (30.775,25.48) circle (4.0pt);
\node[black,anchor=center] at (30.775,25.48) {3};\draw[black,line width=0.8pt,line cap=round] (30.9625,25.0) -- (31.525,25.225);
\draw[black,line width=0.8pt,line cap=round] (31.525,25.225) -- (29.8375,26.215);
\draw[green,line width=0.8pt,fill=white] (30.2125,25.435) circle (4.0pt);
\node[black,anchor=center] at (30.2125,25.435) {1};\draw[black,line width=0.8pt,line cap=round] (29.8375,25.09) -- (30.9625,25.0);
\draw[black,line width=0.8pt,line cap=round] (29.8375,25.09) -- (29.8375,26.215);
\draw[black,line width=0.8pt,line cap=round] (30.9625,25.0) -- (29.8375,26.215);
\draw[black,line width=0.8pt,line cap=round] (32.3125,25.315) -- (33.4375,25.225);
\draw[black,line width=0.8pt,line cap=round] (31.75,25.09) -- (32.3125,25.315);
\draw[black,line width=0.8pt,line cap=round] (32.3125,25.315) -- (31.75,26.215);
\draw[purple,line width=0.8pt,fill=white] (32.425,25.369) circle (4.0pt);
\node[black,anchor=center] at (32.425,25.369) {0};\draw[black,line width=0.8pt,line cap=round] (32.875,25.0) -- (33.4375,25.225);
\draw[black,line width=0.8pt,line cap=round] (33.4375,25.225) -- (31.75,26.215);
\draw[black,line width=0.8pt,line cap=round] (31.75,25.09) -- (32.875,25.0);
\draw[black,line width=0.8pt,line cap=round] (31.75,25.09) -- (31.75,26.215);
\draw[black,line width=0.8pt,line cap=round] (32.875,25.0) -- (31.75,26.215);
\end{scope}
\end{tikzpicture}
\caption{The numbering of the sub-entities that DefElement uses for three-dimensional reference cells.}
\label{reference:3d-cells}
\end{figure}

\subsection{The degree of an element}\label{sec:definition:degree}
In this section, we define the degree of a finite element.
We define two spaces on each of the cells in \cref{def:reference-cells}:
the natural degree $k$ polynomial (or on a pyramid rational polynomial, or \emph{rationomial}) space;
and the complete polynomial space.

\begin{definition}\label{def:polyspaces}
The natural polynomial spaces of degree $k$ on
an interval (\hspace{1pt}$\mathbb{P}^{\interval}_k$),
a triangle (\hspace{1pt}$\mathbb{P}^{\rtriangle}_k$),
a quadrilateral (\hspace{1pt}$\mathbb{P}^{\quadrilateral}_k$),
a tetrahedron (\hspace{1pt}$\mathbb{P}^{\tetrahedron}_k$),
a hexahedron (\hspace{1pt}$\mathbb{P}^{\hexahedron}_k$),
and a prism (\hspace{1pt}$\mathbb{P}^{\prism}_k$), and the natural rationomial space of degree $k$ on
a pyramid (\hspace{1pt}$\mathbb{P}^{\pyramid}_k$) are defined by
\begin{align*}
\mathbb{P}^{\interval}_k&=\operatorname{span}\left\{x^{p_0}\,\middle|\,p_0\in\mathbb{N},\,p_0\leqslant k\right\},\\
\mathbb{P}^{\rtriangle}_k&=\operatorname{span}\left\{x^{p_0}y^{p_1}\,\middle|\,p_0,p_1\in\mathbb{N}_0,\,p_0+p_1\leqslant k\right\},\\
\mathbb{P}^{\quadrilateral}_k&=\operatorname{span}\left\{x^{p_0}y^{p_1}\,\middle|\,p_0,p_1\in\mathbb{N}_0,\,p_0\leqslant k,\,p_1\leqslant k\right\},\\
\mathbb{P}^{\tetrahedron}_k&=\operatorname{span}\left\{x^{p_0}y^{p_1}z^{p_2}\,\middle|\,p_0,p_1,p_2\in\mathbb{N}_0,\,p_0+p_1+p_2\leqslant k\right\},\\
\mathbb{P}^{\hexahedron}_k&=\operatorname{span}\left\{x^{p_0}y^{p_1}z^{p_2}\,\middle|\,p_0,p_1,p_2\in\mathbb{N}_0,\,p_0\leqslant k,\,p_1\leqslant k,\,p_2\leqslant k\right\},\\
\mathbb{P}^{\prism}_k&=\operatorname{span}\left\{x^{p_0}y^{p_1}z^{p_2}\,\middle|\,p_0,p_1,p_2\in\mathbb{N}_0,\,p_0+p_1\leqslant k,\,p_2\leqslant k\right\},\\
\mathbb{P}^{\pyramid}_k&=\operatorname{span}\left\{\frac{x^{p_0}y^{p_1}z^{p_2}}{(1-z)^{p_0+p_1}}\,\middle|\,p_0,p_1,p_2\in\mathbb{N}_0,\,p_0\leqslant k,\,p_1\leqslant k,\,p_2\leqslant k\right\},
\end{align*}
where $\mathbb{N}_0$ is the set of non-negative integers.

For $d\in\mathbb{N}$, the natural polynomial spaces of degree $k$ on
a simplex (\hspace{1pt}$\mathbb{P}^{\simplex{d}}_k$\hspace{-1pt})
and a cube (\hspace{1pt}$\mathbb{P}^{\hypercube{d}}_k$\hspace{-1pt}) are defined by
\begin{align*}
\mathbb{P}^{\simplex{d}}_k&=\operatorname{span}\left\{\prod_{i=0}^{d-1}x_i^{p_i}\,\middle|\,p_0,p_1,\dots,p_{d-1}\in\mathbb{N}_0\textup{ and }\sum_{i=0}^{d-1}p_i\leqslant k\right\},\\
\mathbb{P}^{\hypercube{d}}_k&=\operatorname{span}\left\{\prod_{i=0}^{d-1}x_i^{p_i}\,\middle|\,p_0,p_1,\dots,p_{d-1}\in\mathbb{N}_0\textup{ and }\max_{i=0}^{d-1}p_i\leqslant k\right\},\\
\end{align*}
\end{definition}
\begin{definition}\label{def:complete-polyspace}
For a cell with topological dimension $d$, the complete polynomial space degree $k$ is defined by
\(\mathbb{P}_k=\mathbb{P}^{\simplex{d}}_k\).
For $d=1$ to $d=3$, we note that
\begin{align*}
\mathbb{P}_k&=\begin{cases}
\mathbb{P}^{\interval}_k&\text{if the topological dimension of the cell is 1},\\
\mathbb{P}^{\rtriangle}_k&\text{if the topological dimension of the cell is 2},\\
\mathbb{P}^{\tetrahedron}_k&\text{if the topological dimension of the cell is 3}.
\end{cases}
\end{align*}
\end{definition}

On every cell except the pyramid, the natural polynomials spaces are the spaces spanned by the degree $k$ Lagrange
finite element on the cell. On the pyramid, however, the Lagrange element spans the smaller space
\cite{rationomials,cockburn-fu}
\[
\mathbb{P}^{\pyramid,\text{Lagrange}}_k=\operatorname{span}\left\{\frac{x^{p_0}y^{p_1}z^{p_2}}{(1-z)^{\min(p_0,p_1)}}\,\middle|\,p_0,p_1,p_2\in\mathbb{N}_0,\,p_0+p_2\leqslant k,\,p_1+p_2\leqslant k\right\}.
\]
This close relationship with the Lagrange element inspires the naming convention
we use for our notions of the degree of an element: 
we now define the polynomial and Lagrange subdegrees and superdegrees of a finite element.

\begin{definition}[polynomial subdegree]\label{def:poly-sub}
The polynomial subdegree of a finite element $(\refel,\polyspace,\dualbasis)$ is the largest value of $k$ such that
$$\mathbb{P}_k\subseteq\polyspace.$$
If $\mathbb{P}_0\not\subseteq\polyspace$, then the polynomial subdegree is undefined.
\end{definition}

\begin{definition}[polynomial superdegree]\label{def:poly-super}
The polynomial superdegree of a finite element $(\refel,\polyspace,\dualbasis)$ is the smallest value of $k$ such that
$$\polyspace\subseteq\mathbb{P}_k.$$
If $\forall k\in\mathbb{N}_0$, $\polyspace\not\subseteq\mathbb{P}_k$, then the polynomial superdegree is undefined.
\end{definition}

\begin{definition}[Lagrange subdegree]\label{def:L-sub}
The Lagrange subdegree of a finite element $(\refel,\polyspace,\dualbasis)$ is the largest value of $k$ such that
$$\mathbb{P}^\refel_k\subseteq\polyspace.$$
If $\mathbb{P}^\refel_0\not\subseteq\polyspace$, then the Lagrange subdegree is undefined.
\end{definition}

\begin{definition}[Lagrange superdegree]\label{def:L-super}
The Lagrange superdegree of a finite element $(\refel,\polyspace,\dualbasis)$ is the smallest value of $k$ such that
$$\polyspace\subseteq\mathbb{P}^\refel_k.$$
If $\forall k\in\mathbb{N}_0$, $\polyspace\not\subseteq\mathbb{P}^\refel_k$, then the Lagrange superdegree is undefined.
\end{definition}

We note that on simplex cells the Lagrange and polynomial degrees will always be equal, but
they will often disagree on other cell types.

For many elements, the polynomial superdegree and Lagrange subdegree are not well-behaved:
for example, for serendipity elements \cite{serendipity} on a quadrilateral cell, the Lagrange subdegree is always 1;
and for Arnold--Boffi--Falk elements \cite{mapping2} on a quadrilateral cell, the polynomial superdegree is $2k+2$, where $k$ is the polynomial subdegree, and so jumps by 2 every time the element's
degree is increased.
It is therefore typical to use either the polynomial subdegree or the Lagrange superdegree
as the canonical degree of an element. In general on DefElement, we index elements by their polynomial subdegree, unless this is undefined in which case a sensible alternative is chosen
(for example, bubble elements have an undefined polynomial subdegree, so are indexed using the Lagrange superdegree instead).
Note that for some elements (including N\'ed\'elec first kind and Raviart--Thomas), this indexing is different to that used by some libraries, including FEniCS \cite{dolfinx,basix}.

Let $(\refel,\polyspace,\dualbasis)$ be a finite element
and let $\{p_0,\dots,p_{n-1}\}$ be a basis of $\polyspace$. If the finite element has Lagrange superdegree $k$, then the basis can be represented by
\begin{equation}
\begin{pmatrix}
p_0\\\vdots\\p_{n-1}
\end{pmatrix}
=\mat{C}
\begin{pmatrix}
q_0\\\vdots\\q_{m-1}
\end{pmatrix}
\end{equation}
for some $\mat{C}=\mathbb{R}^{n\times m}$, where $\{q_0,\dots,q_{m-1}\}$ is a basis of $\mathbb{P}^\refel_k$.
Using \cref{eq:dual-inv}, we can now represent the finite element basis functions as
\begin{equation}
\begin{pmatrix}
\phi_0\\\vdots\\\phi_{n-1}
\end{pmatrix}
=
\mat{D}\inv
\mat{C}
\begin{pmatrix}
q_0\\\vdots\\q_{m-1}
\end{pmatrix}.
\end{equation}
Many implementations store the basis functions by computing and storing the matrix of expansion coefficients $\mat{D}\inv\mat{C}\in\mathbb{R}^{n\times m}$
\cite{FIAT,basix,firedrake,fenics-book,dolfinx}.
The basis functions can then be evaluated at a set of points by evaluating the functions $q_0,\dots,q_{m-1}$ at these points, and then
recombining these with the stored matrix of expansion coefficients. As the functions $q_0,\dots,q_{m-1}$ are a basis of the general space $\mathbb{P}^\refel_k$,
the code to evaluate them is not specific for each finite element, and so the same implementation of the spaces can be used for all
elements.
It is common in implementations to take $\{q_0,\dots,q_{m-1}\}$ to be a basis for which recurrence
formulae exist that can be used to quickly evaluate $\{q_0,\dots,q_{m-1}\}$ at a set of points \cite{dubiner1991,spectral/hp,kirby2010singularity}: for example,
the Legendre polynomials (mapped to $[0,1]$ if the reference cells in \cref{def:reference-cells} are used) could be used for interval cells.
In order to ensure that the basis functions can be stored in this way, many finite element implementations implicitly rely on every element having a well-defined Lagrange superdegree.

The definitions of degree can be extended to macroelements by defining spaces on the macroelement that
are equal to those defined in \cref{def:polyspaces} on each sub-cell, with appropriate continuity between sub-cells.

\section{Element variants}\label{sec:variants}
For many elements, it is possible to define a number of different variants of the element. In many cases,
there is no single variant that is always the best variant, and so implementations may support multiple variants.
For example, for Lagrange elements, it is common to use equally spaced points to define low-degree elements, and
mesh storage formats that use Lagrange elements to represent mesh geometry often assume this equal spacing. However,
if equally spaced points are used for high-degree elements, the basis functions will exhibit Runge's phenomenon
and become numerically unstable \cite{runge}, as shown in \cref{fig:runge}. Hence, for higher-degree elements, it is best to use an alternative set of points, such as
Gauss--Lobatto--Legendre (GLL) points or Chebyshev points \cite[chapter 13]{approximation-theory-practice}.

\begin{figure}
\centering
\vspace{-1.5cm}

\begin{tikzpicture}[x=1.8cm,y=1.8cm]
\clip (24.9,23.5) rectangle (30.3,29.5);

\draw[cyan,line width=0.8pt,line cap=round] (25.0,26.265846399999997) -- (27.25,26.265846399999997);
\fill[magenta] (26.285714285714285,26.265846399999997) circle (3pt);
\draw[orange,line width=0.8pt,line cap=round] (25.0,26.265846399999997) .. controls (25.15,20.523658900000004) and (25.3,26.988282799999997) .. (25.45,27.044828799999998);
\draw[orange,line width=0.8pt,line cap=round] (25.45,27.044828799999998) .. controls (25.6,27.1013748) and (25.75,25.475760599999997) .. (25.9,25.8932896);
\draw[orange,line width=0.8pt,line cap=round] (25.9,25.8932896) .. controls (26.05,26.310818599999998) and (26.2,28.712573199999998) .. (26.35,28.5011872);
\draw[orange,line width=0.8pt,line cap=round] (26.35,28.5011872) .. controls (26.5,28.2898012) and (26.65,24.819361500000003) .. (26.8,25.0);
\draw[orange,line width=0.8pt,line cap=round] (26.8,25.0) .. controls (26.95,25.180638499999997) and (27.1,33.92209639999999) .. (27.25,26.265846399999997);

\begin{scope}[shift={(3,26.265846399999997-25.052505981994834)}]
\draw[cyan,line width=0.8pt,line cap=round] (25.0,25.052505981994834) -- (27.25,25.052505981994834);
\draw[orange,line width=0.8pt,line cap=round] (25.0,25.052505981994834) .. controls (25.15,24.52559554650897) and (25.3,25.329238036380193) .. (25.45,25.077853996395923);
\draw[orange,line width=0.8pt,line cap=round] (25.45,25.077853996395923) .. controls (25.6,24.82646995641165) and (25.75,24.636239774592816) .. (25.9,25.1035825985935);
\draw[orange,line width=0.8pt,line cap=round] (25.9,25.1035825985935) .. controls (26.05,25.570925422594186) and (26.2,27.26669163798249) .. (26.35,27.300611085690374);
\draw[orange,line width=0.8pt,line cap=round] (26.35,27.300611085690374) .. controls (26.5,27.334530533398258) and (26.65,25.514925365406807) .. (26.8,25.0);
\draw[orange,line width=0.8pt,line cap=round] (26.8,25.0) .. controls (26.95,24.485074634593193) and (27.1,25.858363762794557) .. (27.25,25.052505981994834);
\fill[magenta] (26.36046162014029,25.052505981994834) circle (3pt);
\end{scope}

\fill[white,opacity=0.8] (-50,-50) rectangle (50,24.9);
\end{tikzpicture}

\vspace{-2.2cm}

\caption{A basis function of a degree 7 Lagrange element on an interval when defined using
equispaced points (left) and GLL points (right). The basis function defined using equispaced
points exhibits Runge's phenomenon: there are large function values far from the marked point
where the basis function is equal to one. Similar issues can be observed when defining the basis on other cell types.}
\label{fig:runge}
\end{figure}
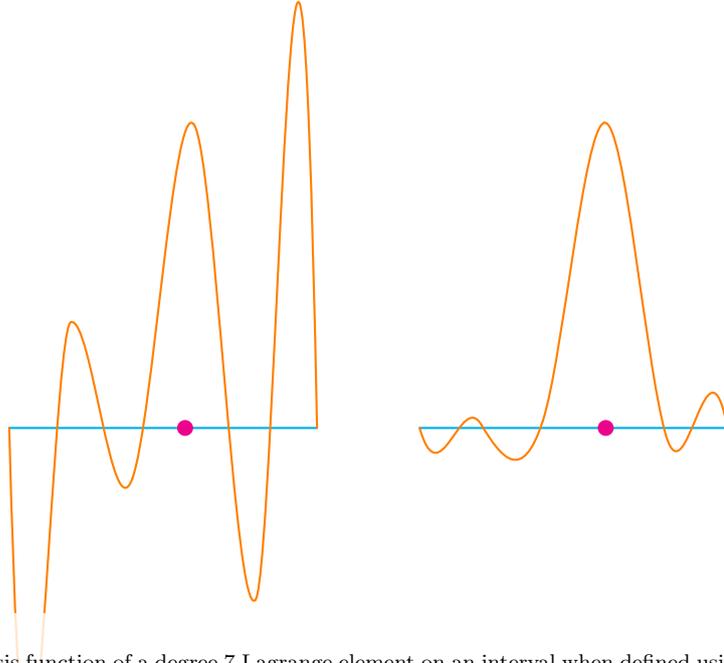

In general, variants of finite elements are defined as follows.

\begin{definition}[Variant of an element]\label{def:variant}
Let $(\refel,\polyspace,\dualbasis)$ and $(\tilde{\refel},\tilde{\polyspace},\tilde{\dualbasis})$ be reference-mapped finite elements.
$(\tilde{\refel},\tilde{\polyspace},\tilde{\dualbasis})$ is a variant of 
$(\refel,\polyspace,\dualbasis)$ if:
\begin{itemize}
\item $\tilde{\refel} = \refel$;
\item $\tilde{\polyspace} = \polyspace$;
\item For all sub-entities $\subentity\subseteq\refel$,
$\dualbasis|_{\subentity}$ and $\tilde\dualbasis|_{\subentity}$ are
equivalent sets of functionals,
where $\dualbasis|_{\subentity}$ and $\tilde\dualbasis|_{\subentity}$ are the sets of functionals in $\dualbasis$
and $\tilde\dualbasis$ that are associated with $\subentity$.
\end{itemize}
\end{definition}

We will now show that a set of alternative criteria is equivalent to \cref{def:variant}, but we must first define
the controlled and uncontrolled trace of a finite element's polynomial space on a sub-entity as follows.

\begin{definition}\label{def:traces}
Let $(\refel,\polyspace,\dualbasis)$ be a reference-mapped finite element. Let $\subentity\subseteq\refel$ be a sub-entity of $\refel$ (i.e.~a vertex, edge, or face of $\refel$, or $\refel$ itself).
Let $\dualbasis = \{l_0,\dots,l_{n-1}\}$ and let $\phi_0,\dots,\phi_{n-1}\in\polyspace$ be the basis functions of the finite element.
The uncontrolled trace space $\unctrace{\polyspace}{\subentity}{\dualbasis}$ is defined by
$$
\unctrace{\polyspace}{\subentity}{\dualbasis} =
\left\{\left.p\right|_{\subentity}\,:\,p\in\polyspace\text{ and }\forall l_i\in\dualbasis|_{\closure{\subentity}}\text{ }l_i(p)=0\right\},
$$
where $\left.\dualbasis\right|_{\closure{\subentity}}$ is the set of all functionals in $\dualbasis$ associated
with the closure of $\subentity$.
The controlled trace space $\ctrace{\polyspace}{\subentity}{\dualbasis}$ is defined by
\begin{align*}
\ctrace{\polyspace}{\subentity}{\dualbasis} &=
\left[
\unctrace{\polyspace}{\subentity}{\dualbasis}
\right]^\perp
\\&=
\left\{\left.p\right|_{\subentity}\,:\,p\in\polyspace\text{ and }\forall q\in\unctrace{\polyspace}{\subentity}{\dualbasis}
\text{ }\left\langle p,q\right\rangle_{\subentity}=0\right\},
\end{align*}
where $\left\langle p,q\right\rangle_{\subentity}:=\int_Ep\cdot q$.
\end{definition}

Note that it is a standard result in functional analysis
(see e.g.~\cite[theorem 3.3-4]{kreyszig}) that if $V$ is a function space and $W$ is a subspace of $V$, then
$V=W\oplus W^\perp$. It follows that
\begin{align}
\left.\polyspace\right|_{\subentity}
&=
\unctrace{\polyspace}{\subentity}{\dualbasis}\oplus\ctrace{\polyspace}{\subentity}{\dualbasis}.
\label{oplus}
\end{align}

As an example, let $(\refel,\polyspace,\dualbasis)$ be a degree $k$ Lagrange element.
For any sub-entity
$\subentity\subset\refel$, the restricted space
$\left.\polyspace\right|_{\subentity}$ is the Lagrange space $\mathbb{P}_k^\subentity$.
The uncontrolled trace space $\unctrace{\polyspace}{\subentity}{\dualbasis}$
is equal to $\{0\}$, as any non-zero polynomial on $\subentity$ would make at least one functional
associated with $\closure{\subentity}$ non-zero.
It follows from \cref{oplus} that 
the controlled trace space $\ctrace{\polyspace}{\subentity}{\dualbasis}$
is $\mathbb{P}_k^\subentity$.

For a less trivial example, let $(\rtriangle[2.3mm],\polyspace,\dualbasis)$ be a degree
1 Raviart--Thomas element \cite{rt} on a triangle and let $\subentity$ be the edge with
vertices at $(0,0)$ and $(1,0)$. The polynomial space for this element is
$$\polyspace=\operatorname{span}\left\{
\begin{pmatrix}1\\0\end{pmatrix},
\begin{pmatrix}x\\0\end{pmatrix},
\begin{pmatrix}y\\0\end{pmatrix},
\begin{pmatrix}0\\1\end{pmatrix},
\begin{pmatrix}0\\x\end{pmatrix},
\begin{pmatrix}0\\y\end{pmatrix},
\begin{pmatrix}x^2\\xy\end{pmatrix},
\begin{pmatrix}xy\\y^2\end{pmatrix}
\right\},$$
and so
$$\left.\polyspace\right|_{\subentity}=\operatorname{span}\left\{
\begin{pmatrix}1\\0\end{pmatrix},
\begin{pmatrix}x\\0\end{pmatrix},
\begin{pmatrix}0\\1\end{pmatrix},
\begin{pmatrix}0\\x\end{pmatrix},
\begin{pmatrix}x^2\\0\end{pmatrix}
\right\}.$$
There are no functionals associated with vertices in this element, and so the only functionals
associated with $\closure{\subentity}$ are the integral moments of the normal components
that are associated with $\subentity$. These functionals will all be zero if applied to functions
that only have a tangential component and so
\begin{align*}
\unctrace{\polyspace}{\subentity}{\dualbasis}&=\operatorname{span}\left\{
\begin{pmatrix}1\\0\end{pmatrix},
\begin{pmatrix}x\\0\end{pmatrix},
\begin{pmatrix}x^2\\0\end{pmatrix}
\right\},\\
\ctrace{\polyspace}{\subentity}{\dualbasis}&=\operatorname{span}\left\{
\begin{pmatrix}0\\1\end{pmatrix},
\begin{pmatrix}0\\x\end{pmatrix}
\right\}.
\end{align*}
For this element, the normal components are continuous between cells:
the controlled space $\ctrace{\polyspace}{\subentity}{\dualbasis}$ is the space of functions that
are continuous between cells, while the uncontrolled space contains the functions that may jump
between cells.

For the remainder of this section, we will aim to prove that if the uncontrolled trace spaces
of two elements are equal for two elements with the same reference cell and polynomial space,
then they are variants of each other
(\cref{lemma:verification}).
We begin by proving some prelimary lemmas that we will use.

\begin{lemma}\label{lemma:first-unc}
Let $(\refel,\polyspace,\dualbasis)$ be a reference-mapped finite element.
If $\subentity$ is a sub-entity of $\refel$, then
\[
\unctrace{\polyspace}{\subentity}{\dualbasis} =
\operatorname{span}\left\{\left.\phi_i\right|_{\subentity}\,:\,l_i\not\in\left.\dualbasis\right|_{\closure{\subentity}}\right\}.
\]
\end{lemma}
\begin{proof}
Let $p\in\unctrace{\polyspace}{\subentity}{\dualbasis}$. As $p\in\left.\polyspace\right|_{\subentity}:=\operatorname{span}\left\{\left.\phi_0\right|_{\subentity},\dots,\left.\phi_{n-1}\right|_{\subentity}\right\}$, we can write
\[p=\sum_{j=0}^{n-1}\alpha_j\left.\phi_j\right|_{\subentity},\]
for some $\alpha_0,\dots,\alpha_{n-1}\in\mathbb{R}$.
If $l_i\in\dualbasis|_{\closure{\subentity}}$, then
\[
0 = l_i(p)
=\sum_{j=0}^{n-1}\alpha_jl_i\left(\left.\phi_j\right|_{\subentity}\right)
=\alpha_i.
\]
This means that $p\in\operatorname{span}\left\{\left.\phi_i\right|_{\subentity}\,:\,l_i\not\in\left.\dualbasis\right|_{\closure{\subentity}}\right\}$,
and so
\[\unctrace{\polyspace}{\subentity}{\dualbasis} \subseteq
\operatorname{span}\left\{\left.\phi_i\right|_{\subentity}\,:\,l_i\not\in\left.\dualbasis\right|_{\closure{\subentity}}\right\}.
\]
From the definition of the uncontrolled trace (\cref{def:traces}), it can be seen that if
$l_i\not\in\left.\dualbasis\right|_{\closure{\subentity}}$ then $\left.\phi_i\right|_{\subentity}\in\unctrace{\polyspace}{\subentity}{\dualbasis}$, and so 
\[\unctrace{\polyspace}{\subentity}{\dualbasis} \supseteq
\operatorname{span}\left\{\left.\phi_i\right|_{\subentity}\,:\,l_i\not\in\left.\dualbasis\right|_{\closure{\subentity}}\right\}.
\]
\end{proof}

\begin{lemma}\label{lemma:zeroth}
Let \(\left.\polyspace\right|_{\subentity}:=\operatorname{span}\left\{\left.\phi_0\right|_{\subentity},\dots,\left.\phi_{n-1}\right|_{\subentity}\right\}\).
Given $f\in\left.\polyspace\right|_{\subentity}$, there are unique functions
$f^\parallel\in\ctrace{\polyspace}{\subentity}{\dualbasis}$
and
$f^\perp\in\unctrace{\polyspace}{\subentity}{\dualbasis}$
such that
\[f=f^\perp+f^\parallel.\]
Moreover, the operator $P_{\subentity}:\left.\polyspace\right|_{\subentity}\to\ctrace{\polyspace}{\subentity}{\dualbasis}$
defined by $P_{\subentity}:f\mapsto f^\parallel$ is a linear projection.
\end{lemma}
\begin{proof}
See \cite[theorem 3.3-4]{kreyszig}, and note that $P_{\subentity}$ is an orthogonal projection
operator, as defined in \cite[section 3.3]{kreyszig}.
\end{proof}

\begin{lemma}\label{lemma:first-c}
Let $(\refel,\polyspace,\dualbasis)$ be a reference-mapped finite element.
If $\subentity$ is a sub-entity of $\refel$, then
\begin{align*}
\ctrace{\polyspace}{\subentity}{\dualbasis} &=
\operatorname{span}\left\{P_{\subentity}\left(\left.\phi_i\right|_{\subentity}\right)
\,:\,l_i\in\left.\dualbasis\right|_{\closure{\subentity}}\right\},
\end{align*}
where \(P_{\subentity}\) is defined as in \cref{lemma:zeroth}.
\end{lemma}
\begin{proof}
Using \cref{oplus}, we see that
\begin{align*}
\unctrace{\polyspace}{\subentity}{\dualbasis}\oplus\ctrace{\polyspace}{\subentity}{\dualbasis}
&=\left.\polyspace\right|_{\subentity}\\
&=
\operatorname{span}\left\{\left.\phi_0\right|_{\subentity},\dots,\left.\phi_{n-1}\right|_{\subentity}\right\}\\
&=
\operatorname{span}\left(\left\{\left.\phi_i\right|_{\subentity}\,:\,l_i\not\in\left.\dualbasis\right|_{\closure{\subentity}}\right\}
\cup
\left\{\left.\phi_i\right|_{\subentity}\,:\,l_i\in\left.\dualbasis\right|_{\closure{\subentity}}\right\}\right).
\end{align*}
Using \cref{lemma:first-unc}, we therefore see that
\begin{align*}
\unctrace{\polyspace}{\subentity}{\dualbasis}\oplus\ctrace{\polyspace}{\subentity}{\dualbasis}
&=
\unctrace{\polyspace}{\subentity}{\dualbasis}
+
\operatorname{span}\left\{\left.\phi_i\right|_{\subentity}\,:\,l_i\in\left.\dualbasis\right|_{\closure{\subentity}}\right\}.
\end{align*}
By \cref{lemma:zeroth}, we know that for all $l_i\in\left.\dualbasis\right|_{\closure{\subentity}}$,
there exists $g_i\in\unctrace{\polyspace}{\subentity}{\dualbasis}$ such that
\(\left.\phi_i\right|_{\subentity}=g_i+P_{\subentity}\left(\left.\phi_i\right|_{\subentity}
\right)\). It follows that
\begin{equation}
\unctrace{\polyspace}{\subentity}{\dualbasis}\oplus\ctrace{\polyspace}{\subentity}{\dualbasis}
=
\unctrace{\polyspace}{\subentity}{\dualbasis}
+
\operatorname{span}\left\{P_{\subentity}\left(\left.\phi_i\right|_{\subentity}\right)\,:\,l_i\in\left.\dualbasis\right|_{\closure{\subentity}}\right\}.
\label{eq:first-sum}\end{equation}
As (for all $l_i\in\left.\dualbasis\right|_{\closure{\subentity}}$) $P_{\subentity}\left(\left.\phi_i\right|_{\subentity}\right)\in\ctrace{\polyspace}{\subentity}{\dualbasis}$, we see due to \cref{lemma:zeroth} that 
\[\unctrace{\polyspace}{\subentity}{\dualbasis}
\cap
\operatorname{span}\left\{P_{\subentity}\left(\left.\phi_i\right|_{\subentity}\right)\,:\,l_i\in\left.\dualbasis\right|_{\closure{\subentity}}\right\}
=\emptyset,\] and so the sum in \cref{eq:first-sum} is disjoint, ie
\begin{align*}
\unctrace{\polyspace}{\subentity}{\dualbasis}\oplus\ctrace{\polyspace}{\subentity}{\dualbasis}
&=
\unctrace{\polyspace}{\subentity}{\dualbasis}
\oplus
\operatorname{span}\left\{P_{\subentity}\left(\left.\phi_i\right|_{\subentity}\right)\,:\,l_i\in\left.\dualbasis\right|_{\closure{\subentity}}\right\}.
\end{align*}
From this, we conclude that
\begin{align*}
\ctrace{\polyspace}{\subentity}{\dualbasis}
&=
\operatorname{span}\left\{P_{\subentity}\left(\left.\phi_i\right|_{\subentity}\right)\,:\,l_i\in\left.\dualbasis\right|_{\closure{\subentity}}\right\}.
\end{align*}
\end{proof}

\begin{lemma}\label{sublemma}
For all $l_i,l_j\in\left.\dualbasis\right|_{\closure{\subentity}}$,
\[l_i\left(P_{\subentity}\left(\left.\phi_j\right|_{\subentity}\right)\right)=\delta_{ij}.\]
\end{lemma}
\begin{proof}
By \cref{lemma:zeroth}, we know that for all $l_i\in\left.\dualbasis\right|_{\closure{\subentity}}$,
there exists $g_i\in\unctrace{\polyspace}{\subentity}{\dualbasis}$ such that
\(\left.\phi_i\right|_{\subentity}=g_i+P_{\subentity}\left(\left.\phi_i\right|_{\subentity}
\right)\). By \cref{def:traces}, $l_j(g_i)=0$ and so
\[
l_j\left(P_{\subentity}\left(\left.\phi_i\right|_{\subentity}\right)\right)
=
l_j\left(\left.\phi_i\right|_{\subentity}-g_i\right)
=
l_j\left(\left.\phi_i\right|_{\subentity}\right)-l_j\left(g_i\right)
=
l_j\left(\left.\phi_i\right|_{\subentity}\right).
\]
The result then follows from \cref{def:ciarlet}.
\end{proof}

\begin{lemma}\label{lemma:first-basis}
Let $(\refel,\polyspace,\dualbasis)$ be a reference-mapped finite element.
If $\subentity$ be a sub-entity of $\refel$, then
\[\left\{P_{\subentity}\left(\left.\phi_i\right|_{\subentity}\right)
\,:\,l_i\in\left.\dualbasis\right|_{\closure{\subentity}}\right\}\]
is a basis of $\ctrace{\polyspace}{\subentity}{\dualbasis}$, where $P_{\subentity}$ is defined as in \cref{lemma:zeroth}.
\end{lemma}
\begin{proof}
By \cref{lemma:first-c}, we know that
\[
\ctrace{\polyspace}{\subentity}{\dualbasis} =
\operatorname{span}\left\{P_{\subentity}\left(\left.\phi_i\right|_{\subentity}\right)
\,:\,l_i\in\left.\dualbasis\right|_{\closure{\subentity}}\right\},
\]
and so there exists a basis of $\ctrace{\polyspace}{\subentity}{\dualbasis}$ that is a subset of
$\left\{P_{\subentity}\left(\left.\phi_i\right|_{\subentity}\right)
\,:\,l_i\in\left.\dualbasis\right|_{\closure{\subentity}}\right\}$.
Suppose that
\[\sum_{\substack{i\text{ s.t.}\\l_i\in\left.\dualbasis\right|_{\closure{\subentity}}}}\alpha_iP_{\subentity}\left(\left.\phi_i\right|_{\subentity}\right) = 0,\]
and let
$l_j\in\left.\dualbasis\right|_{\closure{\subentity}}$. We can see, using \cref{sublemma}, that
\[
0 = l_j(0)
=
l_j\left(\sum_{\substack{i\text{ s.t.}\\l_i\in\left.\dualbasis\right|_{\closure{\subentity}}}}\alpha_iP_{\subentity}\left(\left.\phi_i\right|_{\subentity}\right)\right)
=
\sum_{\substack{i\text{ s.t.}\\l_i\in\left.\dualbasis\right|_{\closure{\subentity}}}}\alpha_il_j\left(P_{\subentity}\left(\left.\phi_i\right|_{\subentity}\right)\right)
=
\alpha_j,
\]
and so the set 
$\left\{P_{\subentity}\left(\left.\phi_i\right|_{\subentity}\right)
\,:\,l_i\in\left.\dualbasis\right|_{\closure{\subentity}}\right\}$
is linearly independent.
\end{proof}

\begin{lemma}\label{lemma:second}
Let $(\refel,\polyspace,\dualbasis)$ and $(\refel,\polyspace,\tilde{\dualbasis})$ be reference-mapped finite elements
where for all sub-entities $\subentity\subseteq\refel$, there are the same number of functionals in $\dualbasis$ and $\tilde{\dualbasis}$ that are associated with $\subentity$.
If
for all sub-entities $\subentity\subseteq\refel$,
$\unctrace{\polyspace}{\subentity}{\dualbasis} = \unctrace{\polyspace}{\subentity}{\tilde{\dualbasis}}$;
then 
for all sub-entities $\subentity\subseteq\refel$,
$\ctrace{\polyspace}{\subentity}{\dualbasis} = \ctrace{\polyspace}{\subentity}{\tilde{\dualbasis}}$,
\end{lemma}
\begin{proof}
This is follows from \cite[theorem 3.3-4]{kreyszig}.
\end{proof}

\begin{lemma}\label{lemma:fourth}
Let $(\refel,\polyspace,\dualbasis)$ and $(\refel,\polyspace,\tilde{\dualbasis})$ be reference-mapped finite elements
where for all sub-entities $\subentity\subseteq\refel$, there are the same number of functionals in $\dualbasis$ and $\tilde{\dualbasis}$ that are associated with $\subentity$.
Let $\subentity\subseteq\refel$ be a sub-entity of $\refel$.
If
$\ctrace{\polyspace}{\subentity}{\dualbasis} = \ctrace{\polyspace}{\subentity}{\tilde{\dualbasis}}$,
then 
$\left.\dualbasis\right|_{\closure{\subentity}}$ and $\tilde{\dualbasis}|_{\closure{\subentity}}$
are equivalent sets of functionals.
\end{lemma}
\begin{proof}
Let $\subentity\subseteq\refel$ by a sub-entity of $\refel$, and suppose that
$\ctrace{\polyspace}{\subentity}{\dualbasis} = \ctrace{\polyspace}{\subentity}{\tilde{\dualbasis}}$.
Let $\mathcal{I}_{\subentity}$ and $\widetilde{\mathcal{I}}_{\subentity}$ be sets of indices such that
$\{l_i:i\in\mathcal{I}_{\subentity}\}=\dualbasis|_{\closure{\subentity}}$ and
$\{\tilde{l}_i:i\in\widetilde{\mathcal{I}}_{\subentity}\}=\tilde{\dualbasis}|_{\closure{\subentity}}$.

By \cref{lemma:first-basis}, we know there exist
$\{p_i:i\in\mathcal{I}_{\subentity}\}\subset\ctrace{\polyspace}{\subentity}{\dualbasis}$
and
$\{q_i:i\in\widetilde{\mathcal{I}}_{\subentity}\}\subset\ctrace{\polyspace}{\subentity}{\tilde{\dualbasis}}$
such that $l_i(p_j)=\delta_{ij}$ and
(for $i,j\in\widetilde{\mathcal{I}}_{\subentity}$)
$\tilde{l}_i(q_j)=\delta_{ij}$.
As $\ctrace{\polyspace}{\subentity}{\dualbasis} = \ctrace{\polyspace}{\subentity}{\tilde{\dualbasis}}$, we see from 
\cref{lemma:first-basis}, 
that 
$\{p_i:i\in\mathcal{I}_{\subentity}\}$ and $\{q_i:i\in\widetilde{\mathcal{I}}_{\subentity}\}$ are bases of
the space $\ctrace{\polyspace}{\subentity}{\dualbasis}$. Hence
there exist $\alpha_{ij}\in\mathbb{R}$ such that
(for $i\in\mathcal{I}_{\subentity}$)
\[
q_i = \sum_{j\in\widetilde{\mathcal{I}}_{\subentity}} \alpha_{ij}p_j.
\]
Using this, we see that (for $j\in\mathcal{I}_{\subentity}$)
\begin{equation}\label{eq:l_j(q_i)}
l_j\left(q_i\right)
=
l_j\left(\sum_{k=a}^b \alpha_{ik}p_k\right)
=
\sum_{k=a}^b \alpha_{ik}l_j\left(p_k\right)
=
\alpha_{ij}.
\end{equation}

Let $f\in\ctrace{\polyspace}{\subentity}{\dualbasis}$. As $\{q_i:i\in\widetilde{\mathcal{I}}_{\subentity}\}$ is a basis, there
exist $c_j\in\mathbb{R}$ such that $f=\sum_{j\in\widetilde{\mathcal{I}}_{\subentity}}c_jq_j$. We see that
(for $i\in\widetilde{\mathcal{I}}_{\subentity}$)
\begin{equation}\label{eq:l_i(f)}
\tilde{l}_i(f)
=
\tilde{l}_i\left(\sum_{j\in\widetilde{\mathcal{I}}_{\subentity}}c_jq_j\right)
=
\sum_{j\in\widetilde{\mathcal{I}}_{\subentity}}c_j\tilde{l}_i\left(q_j\right)
=
c_i.
\end{equation}
Using \cref{eq:l_j(q_i),eq:l_i(f)}, we see that
\begin{align*}
l_j(f)
= \sum_{i=a}^bc_il_j\left(q_i\right)
= \sum_{i=a}^bc_i\alpha_{ij}
= \sum_{i=a}^b\alpha_{ij}l_i(f),
\end{align*}
and so
\[
l_j = \sum_{i=a}^b\alpha_{ij}\tilde{l}_i.
\]
It follows that $\dualbasis|_{\closure{\subentity}}$ and $\tilde{\dualbasis}|_{\closure{\subentity}}$
are equivalent.
\end{proof}

\begin{lemma}\label{lemma:fifth}
Let $(\refel,\polyspace,\dualbasis)$ and $(\refel,\polyspace,\tilde{\dualbasis})$ be reference-mapped finite elements
where for all sub-entities $\subentity\subseteq\refel$, there are the same number of functionals in $\dualbasis$ and $\tilde{\dualbasis}$ that are associated with $\subentity$.
If
for all sub-entities $\subentity\subseteq\refel$,
$\left.\dualbasis\right|_{\closure{\subentity}}$ and $\tilde{\dualbasis}|_{\closure{\subentity}}$
are equivalent sets of functionals;
then for all sub-entities $\subentity\subseteq\refel$,
$\left.\dualbasis\right|_{\subentity}$ and $\tilde{\dualbasis}|_{\subentity}$
are equivalent sets of functionals.
\end{lemma}
\begin{proof}
We can show this by induction on the dimension of the sub-entity.
If the sub-entity is of dimension 0 (i.e.~a vertex), then
$\closure{\subentity}=\subentity$ and the result trivially holds.

If the result holds for sub-entities of dimension $<k$, then let $\subentity$ be an entity
of dimension $k$.
We can write $\closure{\subentity}=\subentity \cup \subentity_0 \cup\dots\cup\subentity_{m-1}$,
where $\subentity_0,\dots,\subentity_{m-1}$ are sub-entities of dimension $<k$.
We can see that
\begin{align*}
\operatorname{span}\dualbasis|_{\closure{\subentity}}
&=
\operatorname{span}\dualbasis|_{\subentity}\oplus
\operatorname{span}\dualbasis|_{\subentity_0}\oplus\dots\oplus
\operatorname{span}\dualbasis|_{\subentity_{m-1}}
\\
&=
\operatorname{span}\dualbasis|_{\subentity}\oplus
\operatorname{span}\tilde{\dualbasis}|_{\subentity_0}\oplus\dots\oplus
\operatorname{span}\tilde{\dualbasis}|_{\subentity_{m-1}},
\end{align*}
and also
\begin{align*}
\operatorname{span}\dualbasis|_{\closure{\subentity}}
&=
\operatorname{span}\tilde{\dualbasis}|_{\closure{\subentity}}
\\&=
\operatorname{span}\tilde{\dualbasis}|_{\subentity}\oplus
\operatorname{span}\tilde{\dualbasis}|_{\subentity_0}\oplus\dots\oplus
\operatorname{span}\tilde{\dualbasis}|_{\subentity_{m-1}}.
\end{align*}
Therefore $\operatorname{span}\tilde{\dualbasis}|_{\subentity}=\operatorname{span}\dualbasis|_{\subentity}$.
By \cref{lemma:equivalent}, it follows that $\dualbasis|_{\subentity}$ and $\tilde{\dualbasis}|_{\subentity}$
are equivalent sets of functionals.
\end{proof}

We can now show the main result of this section.

\begin{theorem}\label{lemma:verification}
Let $(\refel,\polyspace,\dualbasis)$ and $(\refel,\polyspace,\tilde{\dualbasis})$ be finite elements.
If
for all sub-entities $\subentity\subseteq\refel$,
$\unctrace{\polyspace}{\subentity}{\dualbasis} = \unctrace{\polyspace}{\subentity}{\tilde{\dualbasis}}$;
then $(\tilde{\refel},\tilde{\polyspace},\tilde{\dualbasis})$ is a variant of $(\refel,\polyspace,\dualbasis)$.
\end{theorem}
\begin{proof}
The first two bullet points in \cref{def:variant} trivially hold, so we only need to show the
third bullet point.

First, \cref{lemma:second} tells us that the assumptions of the theorem imply that
for all sub-entities $\subentity\subseteq\refel$,
$\ctrace{\polyspace}{\subentity}{\dualbasis} = \ctrace{\polyspace}{\subentity}{\tilde{\dualbasis}}$.
Applying \cref{lemma:fourth} to every sub-entity $\subentity\subseteq\refel$, we see that
$\left.\dualbasis\right|_{\closure{\subentity}}$ and $\tilde{\dualbasis}|_{\closure{\subentity}}$
are equivalent sets of functionals. We can then apply \cref{lemma:fifth} to see that
$\left.\dualbasis\right|_{\subentity}$ and $\tilde{\dualbasis}|_{\subentity}$
are equivalent sets of functionals.
\end{proof}

\Cref{lemma:verification} gives us a way to show that two elements are variants using only their basis functions.
We will use it in the following section as the basis of our algorithm to verify that elements from
different implementations are variants of each other.

\section{Element verification}\label{sec:verification}
There are a large number of open source libraries available that include implementations of
finite elements. Without a significant amount of work by the user, it is difficult to be
certain that the implementations of the same element in different libraries are indeed
the implementations of the same element, and that the way the element is implemented
agrees with the element's definition in the literature. For less simple elements in particular,
it is quite possible that a bug in the implementation leads to an element that ``looks ok''
but does not solve problems to the accuracy that it should.

In order to confirm that implementations of the same element in different libraries are actually
the same element, DefElement performs verification of finite elements. This verification is
performed monthly, with the latest results displayed at \href{https://defelement.org/verification/}{defelement.org/verification/}.
It is currently performed for elements implemented in Basix \cite{basix}, FIAT \cite{FIAT}, and ndelement \cite{ndelement}, with Symfem \cite{symfem} used
as the verification baseline.
DefElement's element verification has led the authors of this paper to
find and correct a number of bugs in FIAT and Symfem.

When verifying elements, we do not require the basis functions in the two implementations to be
exactly the same, as two libraries may use different variants of the same element, as defined in 
\cref{def:variant}. These variants will have different basis functions, but their important properties
that the error of a problem solution depends on will be the same. We therefore designed an algorithm
to verify that two elements are variants of each other: this is shown in \cref{alg:verification}.

\begin{algorithm}
\caption{
Verifying that two finite elements $(\refel_0,\polyspace_0,\dualbasis_0)$ and $(\refel_1,\polyspace_1,\dualbasis_1)$ are variants of each other.
In this algorithm,
$\nelements{S}$ is the cardinality of a set $S$.
}\label{alg:verification}
\newcommand{\VARIABLE}[1]{\texttt{#1}}
\newcommand{\LET}{\textbf{let }}
\newcommand{\INPUT}{\textbf{input }}
\begin{algorithmic}[1]
\STATE \INPUT $(\refel_0,\polyspace_0,\dualbasis_0)$, $(\refel_1,\polyspace_1,\dualbasis_1)$
\IF{$\refel_0\not=\refel_1$}\label{av:pt1:start}
\STATE{update $\polyspace_1$ by applying push forward from $\refel_1$ to $\refel_0$}
\ENDIF\label{av:pt1:end}
\IF{$\polyspace_0\not=\polyspace_1$}\label{av:pt1.5:start}
\RETURN false
\ENDIF\label{av:pt1.5:end}
\FOR{$\subentity\subseteq\refel_0$}\label{av:pt2:start}
\IF{$\nelements{\left.\dualbasis_0\right|_{\subentity}}\not=\nelements{\left.\dualbasis_1\right|_{\subentity}}$}
\RETURN false
\ENDIF
\ENDFOR\label{av:pt2:end}
\FOR{$\subentity\subseteq\refel_0$}\label{av:pt3:start}
\IF{$\unctrace{\polyspace_0}{\subentity}{\dualbasis_0}\not=\unctrace{\polyspace_1}{\subentity}{\dualbasis_1}$}\label{av:unctrace}
\RETURN false
\ENDIF
\ENDFOR\label{av:pt3:end}
\RETURN true
\end{algorithmic}
\end{algorithm}

In \cref{alg:verification}, we first check whether the two implementations of the element have
the same reference cell: if not we apply the push forward map
(lines \ref{av:pt1:start}--\ref{av:pt1:end}). We then check that they
have the same polynomial space
(lines \ref{av:pt1.5:start}--\ref{av:pt1.5:end}).
Next, we check that in each implementation, the same number of DOFs are associated with each sub-entity
of the reference cell
(lines \ref{av:pt2:start}--\ref{av:pt2:end}).
Lastly, we check that the uncontrolled trace spaces for each sub-entity of the reference cell are the same
(lines \ref{av:pt3:start}--\ref{av:pt3:end}).
If these checks all pass, then by \cref{lemma:verification} the two elements are variants of each other.

When implementing this algorithm, however, one cannot simply check that the polynomial spaces in
line \ref{av:pt1.5:start} and the uncontrolled trace spaces in \ref{av:unctrace} are equal. Hence to
test if these spaces are the same, we apply \cref{alg:span} to
the bases of the polynomial spaces $\polyspace_0$ and $\polyspace_1$ (for line \ref{av:pt1.5:start})
and the basis functions not associated with $E$ restricted to $E$ (for line \ref{av:unctrace}).
For the polynomial spaces in line \ref{av:pt1.5:start}, we additionally check that the spaces have the
same dimension.

\begin{algorithm}
\caption{Checking if two sets of functions $\mathcal{A}=\{a_0,\dots,a_{n_a-1}\}$
and $\mathcal{B}=\{b_0,\dots,b_{n_b-1}\}$ span the same space on the entity $\subentity$. In this algorithm,
$s_a$ and $s_b$ are the value sizes of the functions in $\mathcal{A}$ and $\mathcal{B}$
and
$\left[f\right]_k$ is equal to $f$ if $f$ is a scalar-valued function, or the $k$th component of $f$ is $f$ is vector- or tensor-valued.
}\label{alg:span}
\newcommand{\VARIABLE}[1]{\texttt{#1}}
\newcommand{\LET}{\textbf{let }}
\newcommand{\INPUT}{\textbf{input }}
\begin{algorithmic}[1]
\STATE \INPUT $(\mathcal{A},\mathcal{B},E)$
\IF{$s_a\not=s_b$} \label{as:pt1:start}
\RETURN false
\ENDIF\label{as:pt1:end}
\STATE \LET $P=\{p_0,\dots,p_{n_P}\}$ be a set of points on $\subentity$\label{as:let-p}
\STATE \LET $M_a=(m^a_{i,j})\in\mathbb{R}^{n_P\times s_an_a}$\label{as:pt2:start}
\STATE \LET $M_b=(m^b_{i,j})\in\mathbb{R}^{n_P\times s_bn_b}$
\FOR{$0\leqslant i<n_P$}
\FOR{$0\leqslant j<s_a$}
\FOR{$0\leqslant k<n_a$}
\STATE \LET $m_{i,ks_a+j} = \left[a_k(p_i)\right]_j$
\ENDFOR
\FOR{$0\leqslant k<n_b$}
\STATE \LET $m_{i,ks_b+j} = \left[b_k(p_i)\right]_j$
\ENDFOR
\ENDFOR
\ENDFOR\label{as:pt2:end}
\IF{$\operatorname{rank}(M_a)\not=\operatorname{rank}(M_b)$}\label{as:pt3:start}
\RETURN false
\ENDIF\label{as:pt3:end}
\RETURN $\operatorname{rank}(M_a)=\operatorname{rank}\left(\begin{pmatrix}M_a\\M_b\end{pmatrix}\right)$\label{as:pt4}
\end{algorithmic}
\end{algorithm}

\Cref{alg:span} first checks that the two sets of functions have the same value shape
(lines \ref{as:pt1:start}--\ref{as:pt1:end}).
In line \ref{as:let-p}, we then define a set of points on the entity $E$. The number of points
$n_P$ must be large enough that the full behaviour of the highest degree polynomial in the input
functions is captured: for the DefElement verificiation, we use a regular lattice of points with
the number of points in each direction much larger than the degree of the highest degree space
that we verify (with points outside triangle, tetrahedron, prism and pyramid cells removed).
We then evaluate each input function at each of these points, and form two matrices that have
the values for one input function on each row
(lines \ref{as:pt2:start}--\ref{as:pt2:end}).
The row spans of these two matrices will be representative of the spans of the two sets of functions,
so we conclude the algorithm by checking that the two matrices have the same rank
(lines \ref{as:pt3:start}--\ref{as:pt3:end}),
and the rank of the two matrices stacked vertically is the same as the span of one matrix
(line \ref{as:pt4})---this final check confirms that the row span of the two matrices is the same,
as if it was not, the rank of the stacked matrix would be larger.

\section{Contributing to DefElement}\label{sec:editing}
Each element on DefElement is defined by a \texttt{.def} file which contains information about the element in yaml format.
The full encyclopedia is generated from these files by scripts written in Python. The full code used to generate DefElement is available under an MIT
license in the DefElement Github repository \cite{defelement-github}, with code that is not finite element specific separated into a
separate website-build-tools repository for easier reuse by other projects \cite{website-build-tools-github}.

The adding and updating of information on DefElement is done using a variety of Github features.
Changes can be suggested, and errors and bugs can be reported using the issue tracker at \href{https://github.com/DefElement/DefElement/issues}{github.com/DefElement/DefElement/issues}.
More detailed discussion about DefElement and its features can be held on the discussions board at \href{https://github.com/DefElement/DefElement/discussions}{github.com/DefElement/DefElement/discussions}.

More direct proposal of changes to DefElement is encouraged, and can be done by forking the DefElement repository, making the changes, then opening a pull request.
The authors of the paper act as the editors of DefElement and are responsible for checking the correctness of new contributions, both those made by other
editors and people outside the team. This checking is done via the pull request review feature on Github, and this allows the editors to comment on specific lines
have changed and suggest direct corrections to lines for more minor issues like typographical errors.

We strongly encourage readers of this paper to contribute to DefElement, by
opening issues for elements that are not yet included in the encyclopedia,
suggesting other changes, or contributing directly via pull requests.

\subsection{Adding a finite element library}
Readers of this paper who are involved in the development of a finite element library
are encouraged to consider adding their library to the set of libraries that DefElement displays information about.
This can be done by adding a file to the folder \texttt{defelement/implementations},
this file should contain the definition of a subclass of DefElement's \texttt{Implementation} class.
This class includes methods that define how the information for the library will appear on the
element's page.

The subclass of \texttt{Implementation} can optionally include a \texttt{verify} function that
takes
the element's string for the implementation (as set in the \texttt{.def} file),
the reference cell,
the element's degree and
a dictionary of additional parameters (as set in the \texttt{.def} file)
and returns lists of entities associated with each subentity and a function that evaluates the basis functions of the element at some input points. If this function
is implemented, then verificiation will be performed for the library.

A full walkthrough guide to adding a new implementation to DefElement can be found on the DefElement
website at \href{https://defelement.org/adding-an-implementation.html}{defelement.org/adding-an-implementation.html}.
Currently only implementations with a Python interface are supported on DefElement, although we hope
to widen the range of supported languages in the future. Details about making other changes can be found on the
DefElement contributing page at \href{https://defelement.org/contributing.html}{defelement.org/contributing.html}.

\section{Concluding remarks}\label{sec:conclusions}
As the number of finite elements that exist in the literature continues to grow, we believe
that it is increasingly important that information about elements can be found in a concise
and consistent format.
We hope that you find the information we have made available on DefElement useful, and
that you consider contributing to encyclopedia by editing existing information and adding
more details and elements.

\section*{Declarations}

\subsection*{Ethical approval}
Not applicable.

\subsection*{Funding}
Support from the Engineering and Physical Sciences Research Council for
JPD (EP/W026635/1) is gratefully acknowledged.
PDB was supported by EPSRC grant EP/W026260/1.
IM was supported by EPSRC grant EP/W524311/1.

\subsection*{Availability of data and materials}
DefElement's source code is available from the DefElement Github repository at \href{https://github.com/DefElement/DefElement}{github.com/DefElement/DefElement}.
An archived snapshot of the code used the build DefElement and the built HTML files is available on
Zenodo \cite{defelement-zenodo}. This archive is automatically updated every three months.


\bibliography{refs}

\end{document}